\documentclass{amsart}
\usepackage{eurosym}
\usepackage{amsthm,amssymb,amsfonts,amsmath}
\usepackage{dsfont}
\usepackage{xcolor}
\usepackage{graphicx}
\usepackage{enumerate}
\usepackage{hyperref}
\usepackage{verbatim}
\usepackage{caption}
\usepackage{float}
\usepackage{multirow}
\usepackage{booktabs}
\usepackage{subfig}

\graphicspath{{Figs/}}
\restylefloat{table}
\newtheorem{theorem}{Theorem}

\newtheorem{corollary}[theorem]{Corollary}

\newtheorem{definition}[theorem]{Definition}

\newtheorem{lemma}[theorem]{Lemma}

\newtheorem{proposition}[theorem]{Proposition}
\newtheorem{remark}[theorem]{Remark}

\input{tcilatex}

\begin{document}
\title[Rotation number in Arnold maps with noise]{Arnold maps with noise:
Differentiability and non-monotonicity of the rotation number}
\author{L. Marangio}
\address{Femto-ST Institute, Universit\'e de Bourgogne Franche-Comt\'e ,
21000 Dijon, France}
\email{lmarangio@gmail.com}
\author{J. Sedro}
\address{Laboratoire de Math\'ematiques d'Orsay, Univ. Paris-Sud, CNRS,
Universit\'e Paris-Saclay, 91405 Orsay, France (Current address: Laboratoire
de Probabilit\'es, Statistique et Mod\'elisation (LPSM), Sorbonne
Universit\'e, Universit\'e de Paris, 4 Place Jussieu, 75005 Paris, France.)}
\email{sedro@lpsm.paris}
\author{S. Galatolo}
\address{Dipartimento di Matematica, Universit\`a di Pisa, Largo Bruno
Pontecorvo 5, 56127 Pisa, Italy}
\email{stefano.galatolo@unipi.it}
\urladdr{http://pagine.dm.unipi.it/~a080288/}
\author{A. Di Garbo}
\address{Consiglio Nazionale delle Ricerche, Istituto di Biofisica Unit\'a
Operativa di Pisa Via G. Moruzzi 1. 56124 Pisa, Italy 
${{}^\circ}$
1 Pisa Italia}
\email{angelo.digarbo@pi.ibf.cnr.it}
\author{M. Ghil }
\address{Geosciences Department and Laboratoire de M\'et\'eorologie
Dynamique (CNRS and IPSL), Ecole Normale Sup\'erieure and PSL Research
University, Paris, France, and Department of Atmospheric and Oceanic
Sciences, University of California at Los Angeles, Los Angeles, California,
USA }
\email{ghil@lmd.ens.fr}
\date{\today }

\begin{abstract}
Arnold's standard circle maps are widely used to study the quasi-periodic
route to chaos and other phenomena associated with nonlinear dynamics in the
presence of two rationally unrelated periodicities. In particular, the El
Ni\~no--Southern Oscillation (ENSO) phenomenon is a crucial component of
climate variability on interannual time scales and it is dominated by the
seasonal cycle, on the one hand, and an intrinsic oscillatory instability
with a period of a few years, on the other. The role of meteorological
phenomena on much shorter time scales, such as westerly wind bursts, has
also been recognized and modeled as additive noise.

We consider herein Arnold maps with additive, uniformly distributed noise.
When the map's nonlinear term, scaled by the parameter $\epsilon$, is
sufficiently small, i.e. {$\epsilon < 1$}, the map is known to be a
diffeomorphism and the rotation number $\rho_{\omega}$ is a differentiable
function of the driving frequency $\omega$.

We concentrate on the rotation number's behavior as the nonlinearity becomes
large, and show rigorously that $\rho _{\omega }$ is a differentiable
function of $\omega $, even for $\epsilon \geq 1$, at every point at which
the noise-perturbed map is mixing. We also provide a formula for the
derivative of the rotation number. The reasoning relies on linear-response
theory and a computer-aided proof. In the diffeomorphism case of {\ $%
\epsilon <1$,} the rotation number $\rho _{\omega }$ behaves monotonically
with respect to $\omega $. We show, using again a computer-aided proof, that
this is not the case when $\epsilon \geq 1$ and the map is not a
diffeomorphism. This lack of monotonicity for large nonlinearity could be of
interest in some applications. For instance, when the devil's staircase $%
\rho =\rho (\omega )$ loses its monotonicity, frequency locking to the same
periodicity could occur for non-contiguous parameter values that might even
lie relatively far apart from each other.
\end{abstract}

\maketitle
\tableofcontents

\section{Introduction and motivation}

The motivation of the present work is to provide further physical and
mathematical insights into the behavior of the El Ni\~no--Southern
Oscillation (ENSO) phenomenon. ENSO is a dominant component of the climate
system's variability on the time scale of several seasons to several years 
\cite{Neel+98,Phil90} and its accurate prediction for 6--12 months ahead is
of great socio-economic importance \cite{Barn12, GJ98, Latif94}.

\noindent \textit{Arnold map with noise as a climate toy model.} The model
studied herein is a highly idealized one that captures, however, two key
features of interest of the ENSO phenomenon, namely frequency locking and
high irregularity. 
Frequency-locking behavior has been observed in many fields of physics in
general \cite{Bak86, BB82, FKS82} and in several ENSO models in particular 
\cite{Chang96, GZT08, JNG94, JNG96, SG01, Tzip94, Tzip95}. For instance,
some early coupled ocean--atmosphere models that attempted to simulate and
predict ENSO were locked into a two-year or a three-year cycle. Clearly, it
is difficult to predict large El Ni\~{n}o events in the Eastern Tropical
Pacific --- which occur irregularly, every 2--7 years --- with a model that
has a persistent, stable periodicity of two or three years \cite{GR00}.

Frequency-locking, also called mode-locking, is due to nonlinear interaction
between an internal frequency $\omega_{\mathrm{i}}$ of the system and an
external frequency $\omega_{\mathrm{e}}$. In the ENSO case, the external
periodicity is the seasonal cycle, while the internal periodicity is
associated with an oscillatory instability that has been studied extensively
in the absence of the seasonal cycle \cite[and references therein]{Neel+98}.
A simple model for systems with two competing periodicities is the
well-known standard circle map \cite{Arn83}. This map is often called the
Arnold map and it is given by Eq.~\eqref{map} at the beginning of the next
section.

\noindent \textit{Strong nonlinearity.} The deterministic Arnold circle map $%
T_{\tau ,\epsilon }:S^{1}\rightarrow S^{1}$ is defined by 
\begin{equation}
T_{\tau ,\epsilon }(x):=x+\frac{2\pi }{\omega }-\dfrac{\epsilon }{2\pi }\sin
(2\pi x)\mod 1\,.  \label{Arn}
\end{equation}%
In the absence of nonlinear effects ($\epsilon =0$) the behavior of the
deterministic map in Eq.~\ref{Arn} is relatively simple: either the driving
frequency $\omega $ is rational, $\omega =p/q$ with $(p,q)$ integers, and
the dynamics is periodic with period $p$; or $\omega $ is irrational and the
iterates \{$X_{n}$\} of the map fill the whole circle $S^{1}$ densely. For
small nonlinearity, $\epsilon \ll 1$, narrow Arnold tongues develop out of
each rational-periodicity point $(\omega =p/q,\epsilon =0)$. Each such
tongue increases in width with increasing $\epsilon $, while the periodicity
within it stays equal to $p$; see Fig.~\ref{fig:tongues}. As $\epsilon $
exceeds the value $1.0$, the Arnold tongues overlap, and chaotic behavior
sets in \cite{Arn83}.

\begin{figure}[h]
\centering
\includegraphics[width=0.99\textwidth]{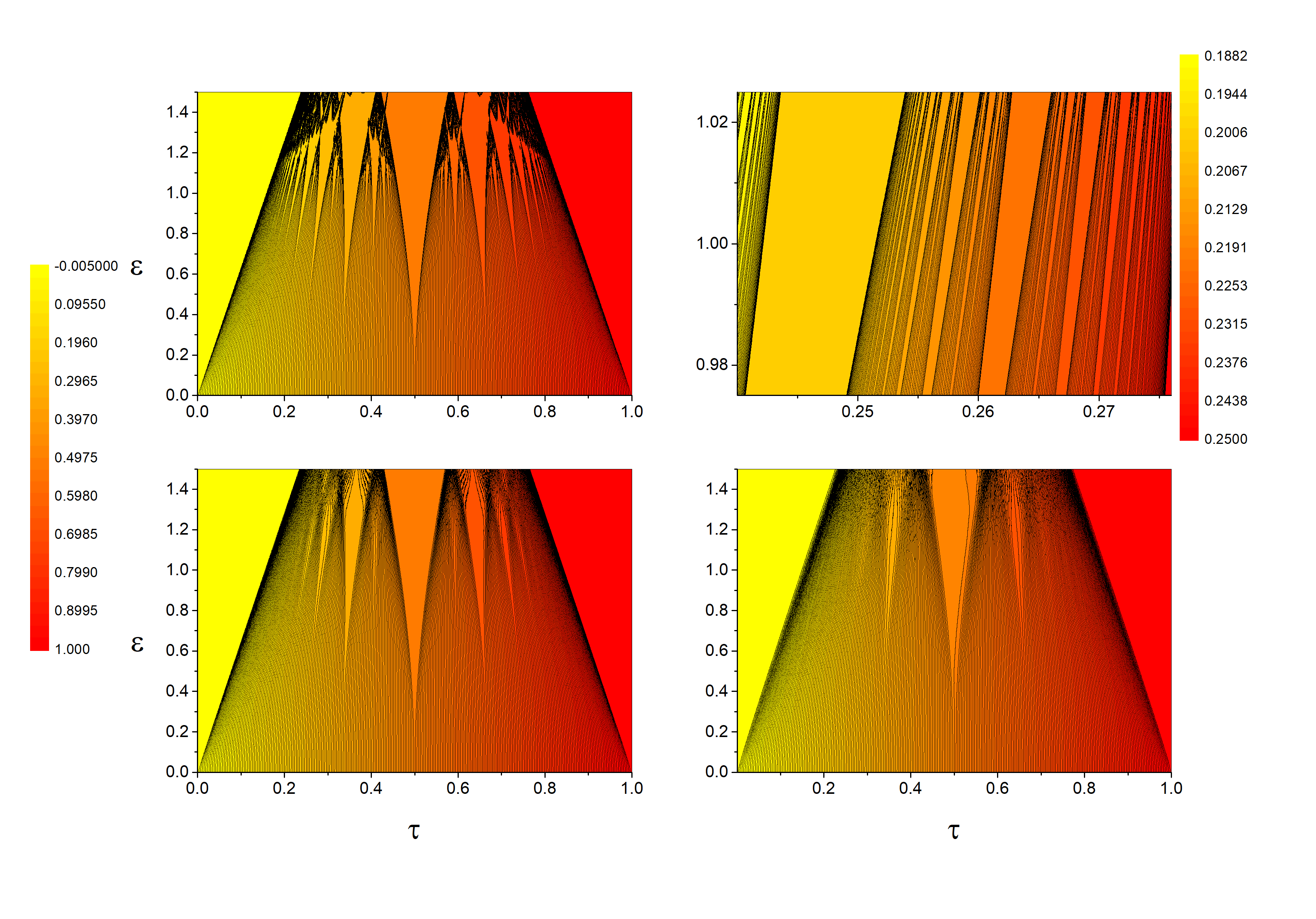}
\caption{Contour maps of the {\ rotation number $\protect\rho = \protect\rho_%
\protect\tau$ in the parameter space $(\protect\tau,\protect\epsilon)$,
where $\protect\tau = 2\protect\pi/\protect\omega$. Each $\protect\rho_%
\protect\tau$ value} was estimated by performing {\ 10~000} iterations of
the map {\ given by Eq.~\eqref{map} in Section~\protect\ref{ssec:model}
below. Top panels: noise amplitude $\protect\xi=0$, with a blow-up near the
critical line $\protect\epsilon=1$ at the right. Bottom panels: Noise
amplitude $\protect\xi=0.02$ (left) and $\protect\xi=0.05$ (right). All
numerical data were obtained for 4~000 values} of $\protect\tau$ and 200 of $%
\protect\epsilon$, uniformly distributed in the corresponding intervals.}
\label{fig:tongues}
\end{figure}

The observed irregular behavior of ENSO argues strongly for large
nonlinearities being active and giving rise to chaotic behavior \cite{GCS08}%
. One such numerical result is illustrated by the ``Devil's bleachers'' in
Fig.~\ref{fig:bleacher}. 
\begin{figure}[h]
\centering
\includegraphics[width=0.7\textwidth,scale=1]{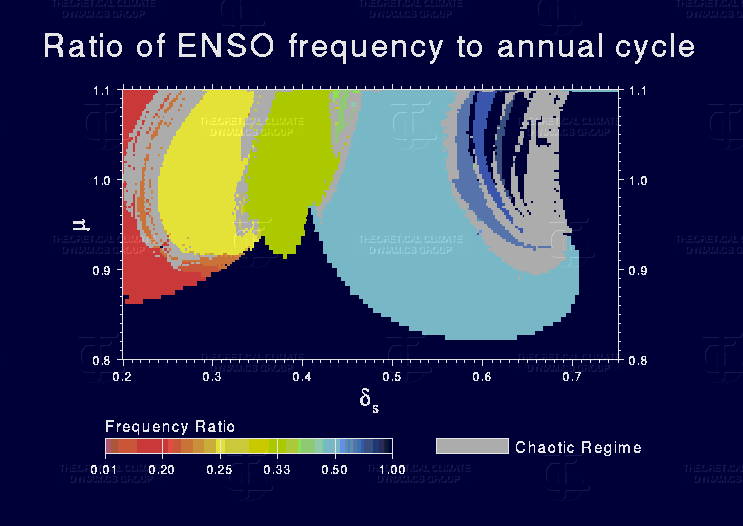}
\caption{Devil's bleachers from a relatively simple ENSO model based on the
coupling of a partial differential equation for the ocean with an integral
equation for the atmosphere \protect\cite{JNG94, JNG96}. The two parameters
are $\protect\delta_{\mathrm{s}}$ on the abscissa and $\protect\mu$ on the
ordinate: $\protect\delta_{\mathrm{s}}$ is a mixed ocean layer parameter
that determines the periodicity of the internal oscillation, while $\protect%
\mu$ corresponds to the strength of the local coupling between the
atmosphere and ocean. The colors identify large areas of ENSO locking to
1--5 years, while gray identifies chaotic behavior; see color bar. {\
Courtesy of Fei-Fei Jin and based on numerical results of \protect\cite%
{JNG96}.}}
\label{fig:bleacher}
\end{figure}

\noindent \textit{Noise sources in ENSO modeling.} In climate modeling in
general, variability on smaller scales in time and space is increasingly
modeled as random. The so-called parametrization --- i.e., large-scale
representation of net effects --- of subgrid-scale phenomena plays an
increasing role in refining the most detailed and highly resolved climate
models \cite{PalWil09}. More specifically, the role of westerly wind bursts
in the onset of El Ni\~{n}os has been studied more and more intensively \cite%
{Ver98}. These wind bursts over the western Tropical Pacific are at least
correlated with and possibly causal to warm events in the eastern Tropical
Pacific \cite{EYuTz05}.

Timmermann and Jin \cite{TJ02} have included a stochastic process meant to
represent these wind bursts into a low-order ENSO model and shown that it
contributes to the irregular occurrence of the model's warm events. This
stochastically perturbed ENSO model has been further studied in \cite{CSG11}%
, where its random attractor has been computed. Using an additive stochastic
process in the toy model studied herein seems therewith amply justified.

{\ In the present paper, we will formulate and apply techniques allowing one
to obtain results on the rotation number of the strongly nonlinear Arnold
circle map in the presence of additive noise. No claim is made that this
results obtained herein could be applied directly to high-end climate
models. It is common, though, in the climate sciences to deal with a full
hierarchy of models \cite{Ghil.2001, Held.2005, Sch.Dick.1974}--- from the
Arnol'd circle map, through low-order systems of ordinary differential
equations \cite{Lor63a},\cite{Lor63b}, \cite{Jal95}, \cite{GhilDCDS}, \cite%
{Pierinial} and on to intermediate models governed by partial differential
equations in one or two space dimensions, all the way to fully
three-dimensional and highly detailed general circulation models. Moreover,
it is more and more the case that the simulations of intermediate and
high-end models are analyzed by data-driven methods to yield
stochastic-dynamic models reproducing their main properties \cite{MCP.1996,
MSM.2015}. It is to these models that advanced mathematical techniques, like
the ones proposed in this paper, are then to be applied. }

As we shall see forthwith, a key tool of our approach is linear response
theory, which has already been applied fairly widely in the climate sciences
(\cite{Luu},\cite{GhLu},\cite{HM}). From the theoretical point of view
linear response theory was proved to apply generally in random systems in
which the presence of noise plays has a regulatizing effect (\cite{HM},\cite%
{GG}). We can expect, therefore, that several of the general ideas presented
herein can be applied to low-order and intermediate climate models, in the
presence of additive noise.


\noindent \textit{Linear response and stability of the statistical
properties.} In the following sections, we study the behavior of the
statistical properties of Arnold maps with additive noise and strong
nonlinearity. In particular, we study the behavior of their rotation number.
We are interested in the smoothness of the rotation number $\rho $ as a
function of the driving frequency $\omega $ and in its monotonicity
properties. We will see that, in the strong nonlinearity case $\epsilon \geq
1$, $\rho =\rho (\omega )$ still varies smoothly, but it is not monotonic
anymore, as it is in the weakly nonlinear case $\epsilon <1$.

Our findings rely both on general mathematical results and on rigorous
computer-aided estimates. The smoothness of statistical properties of a
family of dynamical systems is often guaranteed by the fact that its
relevant stationary or invariant measure 
varies in a smooth way {with respect to changes in a control parameter 
of the system. This property} is called a \emph{linear response} of the
invariant measure under perturbation of the system. The system's linear
response with respect to a perturbation can be described by a suitable
derivative, \emph{representing the rate of change of the relevant (physical,
stationary) invariant measure of the system with respect to the perturbation}%
.

The behavior of the stationary measure of a random or deterministic
dynamical system under perturbations may be very different from system to
system. Some classes of systems have a smooth behavior, with respect to
suitable perturbations; this smoothness was proved for several classes of
deterministic systems. Starting with the work of Ruelle (see \cite{R}) which
proved this for uniformly hyperbolic systems, similar results have been
proved in some cases for non-uniformly expanding or hyperbolic ones (see
e.g. \cite{BahSau, Ba1, BT, BaSma,BKL, D, K,zz}). On the other hand, it is
known that linear response does not always hold, due to the lack of
regularity of the system or of the perturbation or to insufficient
hyperbolicity; see \cite{Ba1, BBS, Gpre, zz}. The survey paper \cite{BB} has
an exhaustive list of classical references on the subject.

An example of linear response for small random perturbations of
deterministic systems appears in \cite{Li2}. Results for random systems were
proved in \cite{HM}, {where} {the} technical framework {was} adapted to
stochastic differential equations and in \cite{BRS}, where the authors
consider random compositions of expanding or non-uniformly expanding maps.
Rigorous numerical approaches for the computation of the linear response are
available, to some extent, both for deterministic and random systems (see 
\cite{BGNN, PV}).

In recent work \cite{GG}, it was shown that, in the presence of additive
noise, one can expect linear response, even for maps which are not expanding
or hyperbolic. The Arnold maps with noise and strong nonlinearity and the
kinds of perturbations we consider herein fit into this framework. In this
paper, we will adapt the results of \cite{GG} to prove linear response for
this class of systems and perturbations, along with smoothness of the
rotation number. We remark that the results of \cite{GG} are not directly
applicable to the kind of perturbations we need to consider here because
these perturbations change the critical values of the deterministic part of
the dynamics. In Section \ref{sec33} we show how to deal with these
perturbations.

\noindent \textit{Computer-aided estimates.} Some of the results we present
have been obtained with the help of computer-aided estimates. These
estimates are obtained using suitable numerical software that tracks the
possible truncation and numerical errors during the computation. The output
of the computation is then an interval containing in a certified way the
result that was meant to be estimated, e.g. \textquotedblleft the rotation
number $\rho $ of the given system is contained in the interval $%
[0.556,0.566]$". Such rigorous computations can be implemented by suitable
numerical methods and libraries, the results can be considered as statements
proved by a computer-aided estimate; see, for instance, the book \cite{Tuc}
for an introduction to the subject.

In this paper we will make use of the software and the methods developed in 
\cite{GMN} for dynamical systems on the interval with additive noise.%
\footnote{\label{notacode}The code used in this paper is at %
\url{https://bitbucket.org/luigimarangio/arnold_map/} .}

The software will be used in this work for two purposes:

\begin{enumerate}
\item computing the stationary measure of a given Arnold map with noise to
within a small, explicit error in the $L^{1}$-norm.

\item computing such a{\ system's mixing rate. The mixing rate} is measured
by the norm of the iterates of the transfer operator associated with the
system, restricted to the space of measures having zero average on the phase
space $S^{1}$.
\end{enumerate}

Both of these computations are made possible by a kind of finite-element
approximation of the transfer operator, used in combination with
quantitative functional analytic stability statements that estimate
explicitly, and not asymptotically the approximation errors made in the
finite element reduction (as explained in \cite{GMN}). Further details on
this matter will be given in Sections \ref{exp1} and \ref{exp2}, where we
also show the results of the computer-aided estimates we use herein.

\noindent \textit{Plan of the paper. } The paper is structured as follows:

In Section \ref{sec2} we introduce the Arnold maps with noise that we study
and the questions which are investigated in the paper. We also state
informally the {\ paper's main results.} 

In Section \ref{sec33} we outline a general linear-response statement, and
adapt it to the kind of systems and perturbations we investigate in the
paper. This will be the core tool to show that the rotation number varies
smoothly even in the strong nonlinear case $\epsilon \geq 1$.

The application of the theory built in Section \ref{sec33} requires some
quantitative estimate on the {system's mixing rate.} In Section \ref%
{comptool}, we prove the quantitative stability results that are required to
support a computer-aided estimation of the {\ system's mixing rate} and show
the result of such computer-aided estimates.

{\ Further computer-aided estimates in Section ~\ref{exp2} yield rigorous
non-monotonicity results for} the rotation number in the strong nonlinearity
case.

\section{Mode locking in the presence of noise and our main results\label%
{sec2}}

In this section we introduce more precisely the systems and the problems
being studied in the paper. We also state informally the main results of the
paper and summarize the findings of the paper \cite{ZH07}, which could be
used to prove the smoothness of the rotation number in the weak nonlinearity
case $\epsilon <1$, but cannot be used in the case $\epsilon \geq 1$.

\subsection{Model formulation and questions investigated\label{ssec:model}}

The system we study is the stochastically perturbed Arnold circle map, where
the usual, deterministic circle map $T_{\tau ,\epsilon }:S^{1}\rightarrow
S^{1}$ is defined by 
\begin{equation}
T_{\tau ,\epsilon }(x):=x+\tau -\dfrac{\epsilon }{2\pi }\sin (2\pi x)\mod 1%
\,.  \label{map}
\end{equation}%
Here $\tau :=2\pi /\omega $ and $\omega :=\omega _{\mathrm{i}}/\omega _{%
\mathrm{e}}$ {\ is the driving frequency,} 
while $\epsilon \geq 0$ parameterizes the magnitude of nonlinear effects. In
other situations, where external driving is replaced by genuine coupling
between two oscillators, one also refers to $\epsilon $ as the coupling
parameter.

By Arnold map with additive noise we mean the stochastic process $%
\{X_{n}\}_{n\in \mathbb{N}}$ on $S^{1}$ defined by 
\begin{equation}
X_{n+1}=T_{\tau ,\epsilon }(X_{n})+\Omega _{n}\mod 1\,.  \label{proc}
\end{equation}%
At each iterate of \ $T_{\tau ,\epsilon }(x)$, an independent identically
distributed (i.i.d.) noise $\Omega _{n},$ that is uniformly distributed on $%
[-\xi /2,\xi /2],$ is added to the deterministic term on the right-hand side
of \eqref{proc}; in particular, the noise is independent of the point $%
X_{n}\in S^{1}$.

In the case $\epsilon =0,\xi =0$, the system is simply a rotation of the
circle. In the deterministic case, where $\xi =0$, and $\epsilon \in (0,1)$
we get the classical Arnold circle map, which is one of the simplest models
of coupled oscillators \cite%
{arnold1965,keener1984global,arnold1991cardiac,glass1991cardiac,GCS08}.
Outside climate science, in the context of cardiac dynamics the circle map
was employed as a model for cardiac arrhythmias \cite%
{arnold1991cardiac,glass1991cardiac} in which the irregular dynamics of
heart pumping is interpreted as arising from the competition of two
pacemakers. Similar studies were carried out in neurophysiology to
investigate the dynamical behavior of a neuron subject to periodic
stimulation \cite{keener1984global}. In addition, the Arnold circle map was
recently used as a model of the sleep-wake regulation cycle \cite%
{bailey2018circle}. In another study a chain of coupled Arnold circle maps
was employed to study the emergence of phase-locking patterns as a function
of the coupling \cite{batista2003mode}.

In the deterministic case, one of the most striking feature of this model is
the \emph{mode-locking} phenomenon: let us consider the rotation number. The
rotation number $\rho =\rho _{\tau }$ measures the average rotation per
iterate of \eqref{map} on $S^{1}$ and is defined as%
\begin{equation*}
\rho _{\tau }:=\lim_{n\rightarrow \infty }\frac{\tilde{T}_{\tau ,\epsilon
}^{n}(x)}{n}
\end{equation*}%
where $\tilde{T}_{\tau ,\epsilon }^{n}:\mathbb{R\rightarrow R}$ is the lift
of $T_{\tau ,\epsilon }$ to $\mathbb{R}$, defined by 
\begin{equation}
\tilde{T}_{\tau ,\epsilon }(x):=x+\tau -\frac{\epsilon }{2\pi }\sin (2\pi x).
\end{equation}

In this case we consider a system with additive noise as in $($\ref{proc}$)$
we define the rotation number by considering a stochastic process $\tilde{X} 
$ on $\mathbb{R}$ defined by $\tilde{X}_{n+1}=\tilde{T}_{\tau ,\epsilon }(%
\tilde{X}_{n})+\Omega _{n}$ and $\rho _{\tau }:=\lim_{n\rightarrow \infty }%
\frac{\tilde{X}_{n}}{n}$.

There are some general classical results concerning the properties of the
rotation number for deterministic and orientation-preserving diffeomorphisms
of the circle that is useful to recall here (further details can be found in 
\cite{herman1977,katok1995}): it is well known that the rotation number is
independent of $x\in S^{1}.$ A very important result is the Denjoy theorem
that states that an orientation-preserving diffeomorphism (for our maps $%
T_{\tau ,\epsilon }$ this corresponds to the weakly nonlinear case $\epsilon
<1$) having an irrational rotation number is topologically conjugate to an
irrational rotation in $S^{1}$ (see \cite{katok1995,wiggins2003}). The mode
locking corresponds to the fact that the rotation number is locally constant
around rational values of the driving frequency $\tau $: the map "rotation
number vs driving frequency" is a devil's staircase (see Fig.~\ref{Fig1} for
one example). Concerning the existence of the derivative of $\rho _{\tau }$,
with respect to the parameter $\tau $ of an orientation-preserving
diffeomorphism, a classical result asserts that this derivative is defined
in a point $\tau _{0}$ when the corresponding rotation number is irrational (%
$\rho _{\tau _{0}}\not\in \mathbb{Q}$) \cite{herman1977}. Moreover, as shown
recently in \cite{matsumoto2012}, when $\tau _{0}$ is the endpoint of an
interval for which $\rho _{\tau _{0}}\in \mathbb{Q}$ then the following
result holds $\lim_{\tau \rightarrow \tau _{0}}\sup \frac{\rho _{\tau }-\rho
_{\tau _{0}}}{\tau -\tau _{0}}=\infty $. Hence in the diffeomorphism case, $%
\rho _{\tau }$ cannot increase differentiably as $\tau $ grows. Other
results in this direction appear in \cite{villanueva2008}.

In Fig.~\ref{Fig1}, we plot the behavior or the rotation number of the
system with additive noise at each iterate like in \eqref{proc}. It is easy
to notice that in this case, the graph of the rotation number seems to go
through a "smoothing" process, i.e the map $\tau \mapsto \rho _{\tau }$
becomes smooth. In the case of weak nonlinearity, $\epsilon \in (0,1),$ this
was rigorously proved by the work of \cite{ZH07}. \ To the best of our
knowledge, both in the deterministic case and in the case with additive
noise, no rigorous results about the differentiability of $\rho _{\tau }$
are present in the literature for the case where\ $\epsilon \geq 1$.\newline
In the present paper we focus to this case, where the methods of \cite{ZH07}
cannot work; from the physical point of view this case corresponds to strong
nonlinear behavior and from the mathematical point of view it corresponds to
the fact that the map $T_{\tau ,\epsilon }$ is not a diffeomorphism anymore. 
\begin{figure}[h]
\centering
\includegraphics[scale=0.5]{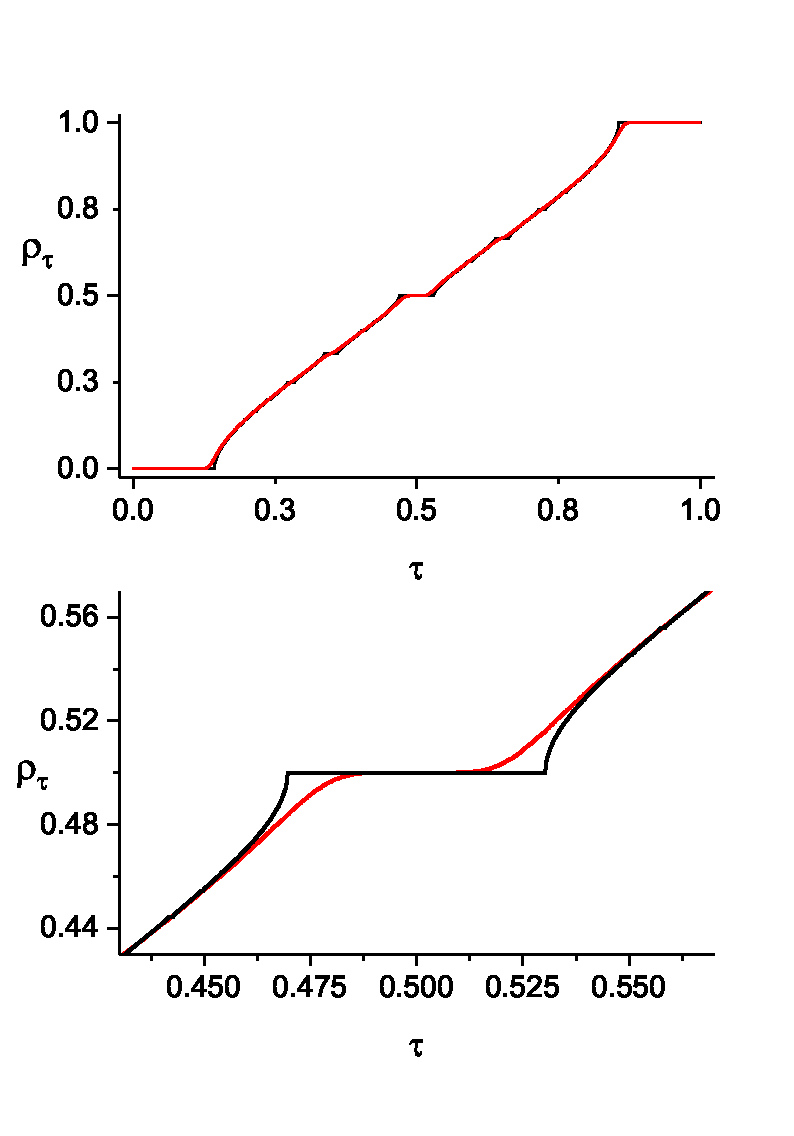}
\caption{ Plot of the rotation number $\protect\rho =\protect\rho _{\protect%
\tau }$ for $\protect\epsilon =0.9$. The black line corresponds to the
absence of uniformly distributed noise, while the red line shows this
dependence in the presence of such noise, with amplitude $\protect\xi =0.05$%
. Recall that $\protect\tau =2\protect\pi /\protect\omega $, where $\protect%
\omega $ is the driving frequency. The lower panel is a blow-up of the
dependence near the value $\protect\tau =0.5$.}
\label{Fig1}
\end{figure}

\begin{figure}[h!]
\caption{Same as Fig.~\protect\ref{Fig1} but for $\protect\epsilon=1.4$ and
for noise amplitude $\protect\xi=0.01$. Here the lower panel is a blow-up of
the dependence near $\protect\tau = 0.7$.}
\label{Fig2}\centering
\includegraphics[scale=0.55]{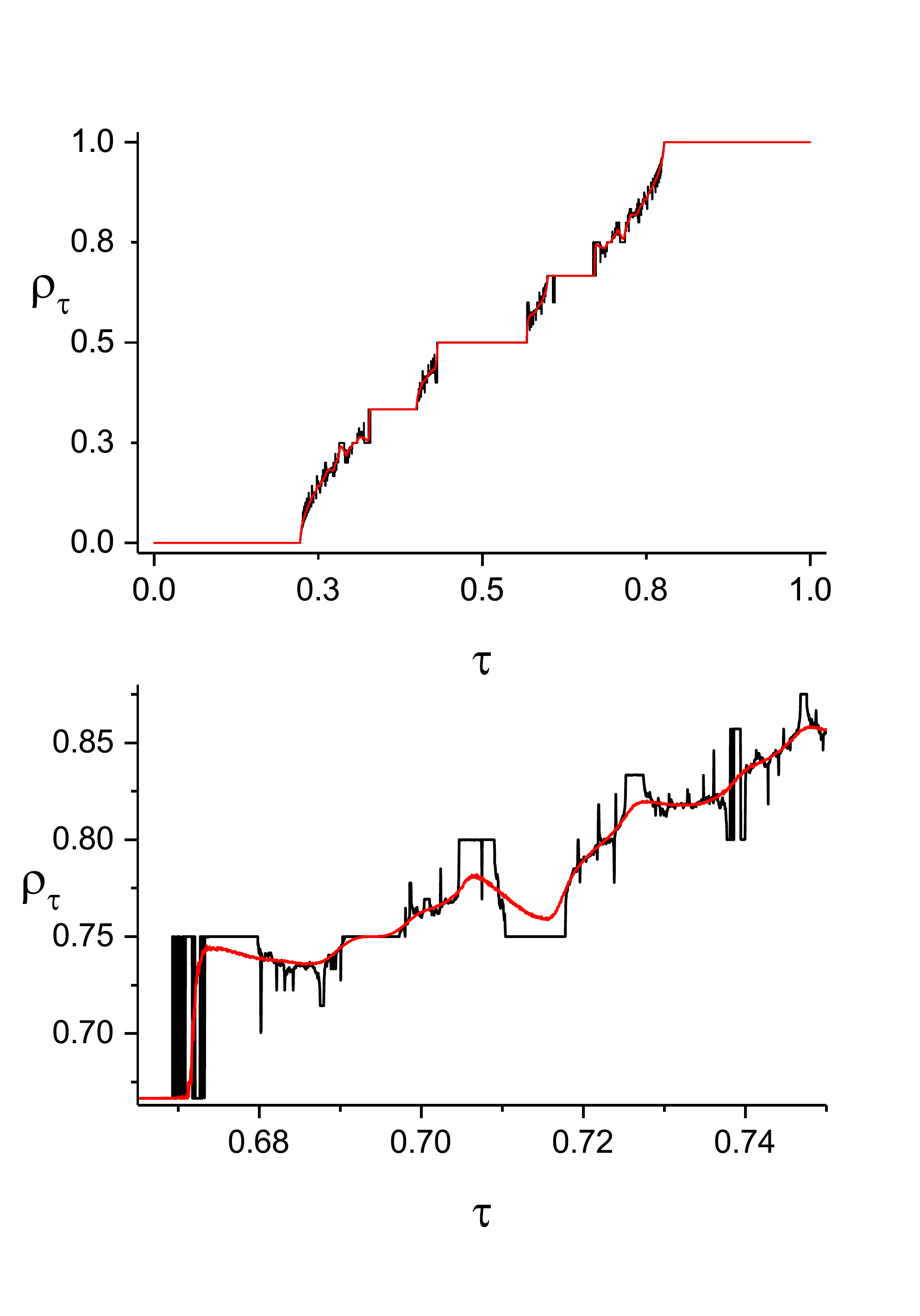}
\end{figure}

For $T_{\tau ,\epsilon }$, when $\epsilon \geq 1$ a mode locking phenomenon
can still be observed, but the behavior of the rotation number as a function
of the driving frequency is more complicated. In Fig.~\ref{Fig2} we show a
plot of the rotation number both in the case with noise and without noise.
We can observe that the action of the noise makes the rotation number
smoother, as in the diffeomorphism case. Furthermore we can observe that the
rotation number is not monotonic anymore as $\tau $ increases. This happens
both with noise and without noise. The goal of this paper is to propose an
explanation for those observations in the presence of noise, at the
crossroads of linear response theory and computer-assisted proofs.

More precisely, we prove that :

\begin{description}
\item[a)] The rotation number is differentiable even in the case $\epsilon
\geq 1$ at every value for the parameter $\tau $ for which the system is
mixing and we provide examples of intervals for which this assumption is
satisfied (see Section \ref{exp1}). We also provide an explicit formula to
compute the derivative.

This will be done adapting some general linear response results for systems
with noise coming from \cite{GG}. In those results the mixing assumption is
needed. We verify the assumption with some certified computer aided
estimates, establishing that the system is mixing when $\tau $ is in certain
intervals.

\item[b)] The rotation number is not always monotonic. In particular we show
intervals for which there is a decreasing of the rotation number. (see
Section \ref{exp2}).

This will be done by a certified approximation of the rotation number for a
certain values of the parameters. The certified estimate on the rotation
number will come from a certified approximation of the stationary measure of
the system with a small error in $L^{1}$, using the framework developed in 
\cite{GMN}.
\end{description}

We emphasize that the tools we develop here are only suitable to study this
problem \emph{in the presence of noise}. More precisely, we can study the
influence of changes in the driving frequency $\tau $, or in the noise
amplitude $\xi $ (although we do not write it here, see \cite[\S 4.2]{GG}),
as long as it does not go to $0$. The approach here used is adapted to
systems which are fastly converging to equilibrium and this is granted in
our examples by the presence of noise.

\section{Differentiability of the rotation number {\ for strong nonlinearity}
\label{sec33}}

In this section, we show that {when the system is mixing, cf. Assumption LR1
of Theorem \ref{th:linearresponse},} there is linear response and the
rotation number is differentiable with respect to changes in the parameter $%
\tau$. {\ This important inference} will be obtained by adapting general
linear response results for systems with additive noise to our case. In the
following section, we start to introduce these results and the preliminary
functional analytic work that is necessary.

The theoretical work will lead to linear response statements for families of
Arnold maps with noise, in which the forcing parameter $\tau $ is modulated.
The linear response will, in turn, lead to the differentiability of the
associated rotation number; see Proposition~\ref{proplinresp} and Corollary~%
\ref{corfin}.

\subsection{Linear response for {\ mixing systems with additive noise}}

\label{ssec:LR_mix} 
Let $BS(S^{1})$ be the set of Borel signed measures on $S^{1}$. This is a
normed vector space if we consider the Wasserstein-Kantorovich norm defined
on $BS$ as%
\begin{equation}
\Vert \mu \Vert _{W}=\underset{\|g\|_{\infty }\leq 1,~Lip(g)\leq 1}{\sup }%
\int_{S^{1}}g(x)d\mu .  \label{Wass}
\end{equation}%
where $Lip(g)$ is the best Lipschitz constant of $g$. We remark that $BS$ is
not complete with this norm. The completion leads to a distributions space
that is the dual of the space of Lipschitz functions. \newline
We recall the following classic fact: the subspace spanned by Dirac masses
in $BS$ is dense for the weak-$\star$ topology. Notice also that in this
topology, one may approximate Dirac masses by multiples of indicators
functions (e.g $\dfrac{1}{2\epsilon}\mathds{1}_{(a-\epsilon,a+\epsilon)}$
goes to $\delta_a$, the Dirac mass at $a\in S^1$, when $\epsilon\to0$, in
the weak-$\star$ topology). Thus, one may approximate any Borel signed
measure by a $L^1$ function in the weak-$\star$ topology. As the unit ball
of the space of Lipschitz functions is compact in $C^0(S^1)$, it follows
that one may approximate any Borel signed measure by a $L^1$ function in the
Wasserstein norm.

\medskip

We remark that since $S^{1}=\mathbb{R}/\mathbb{Z}$ is an additive group
endowed with a Haar-Lebesgue measure (here denoted by $m$), there is a well
defined notion of convolution. More precisely, given $f,g\in L^{1}((S^{1},m),%
\mathbb{R)}$ let us define the convolution $f\ast g$ as 
\begin{equation*}
f\ast g(t)=\int_{\mathbb{S}^{1}}f(t-\tau )g(\tau )d\tau .
\end{equation*}%
It is easy to see that $f\ast g$ is a $L^{1}((S^{1},m),\mathbb{R})$
function, and that the convolution $\ast $ is commutative. In this context,
it is also possible to define the convolution of two signed measures $f,g\in
BS(S^{1})$, $f\ast g\in BS(S^{1})$ as the Borel signed measure 
\begin{equation*}
f\ast g(A):=\int_{S^{1}}\int_{S^{1}}\mathds{1}_{A}(x+y)df(x)dg(y).
\end{equation*}%
There is a particularly useful special case of the convolution of two
measures in $BS(S^{1})$: suppose $f$ is absolutely continuous with respect
to $m$. In this case we have 
\begin{equation*}
f\ast g(t)=\int_{\mathbb{S}^{1}}f(t-\tau )dg(\tau ).
\end{equation*}%
It is noteworthy that in this case the convolution $f\ast g$ is a $L^{1}(m)$
function: this is a first instance of the \emph{regularization} properties
of the convolution, which we highlight and develop in the next section.

\begin{remark}
When dealing with measures which are absolutely continuous with respect to
the Lebesgue measure we will often identify the measure itself with his $%
L^{1}$density, to simplify notations.
\end{remark}

We end by recalling the definition of variation for a signed measure and a
function.

\begin{definition}
\label{def:spaces} {\ Let $g$ be} a finite Borel measure with sign on $S^{1}$%
. Given its decomposition as the difference of two positive measures $%
g=g^{+}+g^{-}$, we define its total variation as 
\begin{equation*}
\Vert g\Vert =g^{+}(S^{1})+g^{-}(S^{1}).
\end{equation*}%
Let $f\in L^{1}(S^{1},m)$ be the density of an absolutely continuous measure
and $P=P(x_{1},x_{2},\ldots ,x_{n})$ be the set of endpoints of a finite
partition of $S^{1}$. Let us denote $x_{0}:=x_{n}$. We define the variation
of $f$ with respect to $P$ as 
\begin{equation}
Var_{P}(f):=\sum_{k=1}^{n}|f(x_{k})-f(x_{k-1})|.
\end{equation}%
If there exists $M$ such that $Var(f):=\underset{P}{\sup }\ Var_{P}\leq M$,
with $P\in \{finite~partitions\}$, then $f$ is said to be of Bounded
Variation. Let the Banach space of Borel measures having a bounded variation
density be denoted as 
\begin{equation*}
BV(S^{1})=\{f\in L^{1},Var(f)<\infty \}
\end{equation*}%
with the norm $\Vert f\Vert _{BV}=\Vert f||_{L^{1}}+Var(f)$. We will always
use $BV$ for $BV(S^{1})$ unless $BV(\cdot )$ specifies an argument for the
space.
\end{definition}

\subsubsection{Regularization estimates\label{Reg}}

{The next lemmas} provide regularization properties of the convolution that
will be used later.

\begin{lemma}
\label{lemma1}Let $f\in L^{1}(m)$ and $g\in BS(S^{1})$. {\ The following
inequality holds:} 
\begin{equation*}
\Vert f\ast g||_{L^{1}}\leq \Vert g\Vert \cdot \Vert f||_{L^{1}}.
\end{equation*}
\end{lemma}

\begin{proof}
Suppose first that $g$ is positive 
\begin{eqnarray*}
\Vert f\ast g||_{L^{1}} &\leq &\underset{\Vert h\Vert _{\infty }\leq 1}{\sup 
}\left\vert \int_{\mathbb{S}^{1}}h(t)~(f\ast g)(t)dt\right\vert  \\
&=&\underset{\Vert h\Vert _{\infty }\leq 1}{\sup }\left\vert \int_{\mathbb{S}%
^{1}}h(t)\int_{\mathbb{S}^{1}}f(t-\tau )dg(\tau )\,dt\right\vert  \\
&\leq &\underset{\Vert h\Vert _{\infty }\leq 1}{\sup }\left\vert \int_{%
\mathbb{S}^{1}}\int_{\mathbb{S}^{1}}h(t)f(t-\tau )dt~dg\right\vert \, \\
&\leq &\int_{\mathbb{S}^{1}}\Vert f\Vert _{L^{1}}~dg(\tau )\leq \Vert f\Vert
_{L^{1}}\cdot g(S^{1}).
\end{eqnarray*}

The general case follows by the linearity of the convolution considering $%
g=g^{+}-g^{-}$ and the fact that $g^{-},g^{+}$ are positive measures.
\end{proof}

\begin{lemma}
\label{convoo copy(1)} Let $f\in BS[S^{1}]$ and $g\in L^{1}(S^{1})$. We have 
\begin{equation}
\Vert f\ast g\Vert _{W}\leq \Vert f\Vert _{W}\cdot \Vert g||_{L^{1}}.
\label{conv1}
\end{equation}
\end{lemma}

\begin{proof}
Let us consider $f\in BS(S^{1})$ and $g\in L^{1}(S^{1})$. Let $h\in
Lip(S^{1})$ be such that $\Vert h\Vert _{\infty }+Lip(h)\leq 1$. \newline
To prove \eqref{conv1}, we estimate the quantity 
\begin{align*}
\left\vert \int_{S^{1}}h(z)d(f\ast g)(z)\right\vert & =\left\vert
\int_{S^{1}}\int_{S^{1}}h_{y}(x)df(x)g(y)dy\right\vert  \\
& \leq \int_{S^{1}}\left\vert \int_{S^{1}}h_{y}(x)df(x)\right\vert |g(y)|dy
\\
& \leq \Vert f\Vert _{W}\Vert g||_{L^{1}}
\end{align*}%
where $h_{y}(x):=h(x+y)$ is also a Lipschitz function satisfying $\Vert
h_{y}\Vert _{\infty }+Lip(h_{y})\leq 1$. \newline
Thus we have 
\begin{equation*}
\Vert f\ast g\Vert _{W}:=\sup_{\QATOP{h\in Lip(S^{1})}{\Vert h\Vert
_{Lip}\leq 1}}\left\vert \int_{S^{1}}h(z)d(f\ast g)(z)\right\vert \leq \Vert
f\Vert _{W}\Vert g||_{L^{1}}
\end{equation*}%
\end{proof}

The next definition introduces the space of zero average measures.

\begin{definition}
\label{v}Let us define the space of zero average measures $V\subset L^{1}$
as 
\begin{equation}
V:=\{f\in L^{1}(S^{1})~s.t.~\int_{S^{1}}f~dm=0\}.  \label{d1}
\end{equation}
\end{definition}

\begin{lemma}
\label{lemmconv2}Let $f\in BS(S^{1})$ such that $f(S^{1})=0$ and $g\in BV$.
Then the convolution $f\ast g\in L^{1}(S^{1})$ and 
\begin{equation}
\Vert f\ast g||_{L^{1}}\leq 2\Vert f\Vert _{W}\cdot \Vert g\Vert _{BV}.
\label{conv2}
\end{equation}
\end{lemma}

\begin{proof}
Consider a $C^{1}$ function $g_{\varepsilon }$ such that $\Vert
g_{\varepsilon }-g||_{L^{1}}\leq \varepsilon $ and $\Vert g_{\varepsilon
}\Vert _{BV}=\Vert g\Vert _{BV}$, then by Lemma \ref{lemma1} 
\begin{equation*}
\Vert f\ast g-f\ast g_{\varepsilon }||_{L^{1}}\leq \varepsilon 
\end{equation*}%
thus we can replace $g$ with $g_{\varepsilon }$ up to an error which is as
small as wanted in the estimate we consider. We now consider $\Vert f\ast
g_{\varepsilon }||_{L^{1}}$. Since $g_{\varepsilon }$ is $C^{1}$ with
compact support, it is bounded and has bounded derivative. Hence there is $C$
such that for every $f_{\epsilon }\in V$ satisfying $\Vert f_{\epsilon
}-f\Vert _{W}\leq $ $C\epsilon $, it holds $|f\ast g_{\varepsilon
}(x)-f_{\epsilon }\ast g_{\varepsilon }(x)|\leq \epsilon $ for every $x\in {%
\mathbb{R}}$, by which 
\begin{equation*}
\Vert f_{\epsilon }\ast g_{\epsilon }-f\ast g_{\varepsilon }||_{L^{1}}\leq
\varepsilon 
\end{equation*}%
Thus we may also replace $f$ with $f_{\epsilon }$ in our main estimate.
Remark that $f_{\epsilon }\ast g_{\varepsilon }(t)$ is a $C^{1}$ function,
with $(f\ast g_{\varepsilon })^{\prime }=f\ast g_{\epsilon }^{\prime }$. 
\newline
Let us consider 
\begin{equation*}
\hat{h}(t)=\left\{ 
\begin{array}{c}
\frac{|f_{\epsilon }\ast g_{\varepsilon }(t)|}{f_{\epsilon }\ast
g_{\varepsilon }(t)}~\mathrm{if}~(f_{\epsilon }\ast g_{\varepsilon })(t)\neq
0 \\ 
0~\mathrm{if}~(f_{\epsilon }\ast g_{\varepsilon })(t)=0%
\end{array}%
\right. 
\end{equation*}%
and 
\begin{equation*}
h(t)=\hat{h}(t)-\int \hat{h}(t)dt.
\end{equation*}%
$\hat{h}$ is simply the sign function of $f_{\epsilon }\ast g_{\epsilon }$.
Thus by definition of $\hat{h}$, 
\begin{equation*}
\Vert f_{\epsilon }\ast g_{\varepsilon }||_{L^{1}}=\int_{\mathbb{S}%
^{1}}\left\vert (f_{\epsilon }\ast g_{\varepsilon })(t)\right\vert
dt=\int_{S^{1}}h(t)(f_{\epsilon }\ast g_{\varepsilon })(t)dt.
\end{equation*}%
Indeed, recalling that $f_{\epsilon }\ast g_{\varepsilon }\in V$ 
\begin{eqnarray*}
\int_{S^{1}}[\hat{h}(t)-\int_{S^{1}}\hat{h}dt](f_{\epsilon }\ast
g_{\varepsilon })(t)dt &=&\int_{S^{1}}\hat{h}(t)(f_{\epsilon }\ast
g_{\varepsilon })(t)dt \\
&=&\int_{\mathbb{S}^{1}}\left\vert (f_{\epsilon }\ast g_{\varepsilon
})(t)\right\vert dt.
\end{eqnarray*}

Taking a one-periodic representative of $f_{\epsilon }\ast g_{\epsilon }$,
we can interpret the integral on $S^{1}$ as an integral on $[0,1]$. Since $h$
has zero average, its primitive $\int_{0}^{t}h(s)ds=k(t)$ is also a
one-periodic map. Applying integration by parts thus yields: 
\begin{equation*}
\Vert f_{\epsilon }\ast g_{\varepsilon
}||_{L^{1}}=\int_{0}^{1}\int_{0}^{t}h(r)dr~(f_{\epsilon }\ast g_{\varepsilon
})^{\prime }(t)dt=\int_{0}^{1}k(t)(f_{\epsilon }\ast g_{\varepsilon
}^{\prime })(t)dt.
\end{equation*}%
Note that since its derivative $h$ is bounded by $2$, $k$ is a $2-$Lipschitz
function on $S^{1}$. \newline
The last integral can be rewritten as 
\begin{align*}
\int_{0}^{1}k(t)\left( \int_{S^{1}}g_{\varepsilon }^{\prime }(t)f_{\epsilon
}(t-s)ds\right) dt& =\int_{S^{1}}\int_{S^{1}}k(t)g_{\varepsilon }^{\prime
}(s)f_{\epsilon }(t-s)dsdt \\
& =\int_{S^{1}}\int_{S^{1}}k(t)f_{\epsilon }(t-s)g_{\varepsilon }^{\prime
}(s)dtds \\
& =\int_{S^{1}}\left( \int_{S^{1}}k(t+s)f_{\epsilon }(t)dt\right)
g_{\epsilon }^{\prime }(s)ds
\end{align*}%
by applying Fubini-Tonelli theorem. Thus we obtain 
\begin{equation*}
\Vert f_{\epsilon }\ast g_{\varepsilon }||_{L^{1}}\leq
\int_{S^{1}}\left\vert \int_{S^{1}}k_{s}(t)f_{\epsilon }(t)dt\right\vert
|g_{\epsilon }^{\prime }(s)|ds\leq 2\Vert f_{\epsilon }\Vert _{W}\Vert
g_{\epsilon }\Vert _{BV}
\end{equation*}%
since $k_{s}:t\mapsto k(t+s)$ is a $2$-Lipschitz function satisfying $\Vert
k_{s}\Vert _{\infty }\leq 2$ for all $s\in S^{1}$. This last estimate being
valid for each $\varepsilon $, the proposition is established.
\end{proof}

\begin{lemma}
\label{lemmaconv4}Let $f\in L^{1},$ $g\in BV$ 
\begin{equation}
\Vert f\ast g\Vert _{BV}\leq \Vert f||_{L^{1}}\cdot \Vert g\Vert _{BV}.
\label{conv3}
\end{equation}
\end{lemma}

\begin{proof}
Similar estimates are well known for the convolution on $\mathbb{R}$. We
prove the estimate on $S^{1}.$ Let us suppose first that $f,g\in C^{1}(S^{1})
$. In this case $f\ast g\in C^{1}$and 
\begin{equation*}
Var(f\ast g)=\int_{\mathbb{S}^{1}}\left\vert (f\ast g)^{\prime
}(t)\right\vert dt=\int_{\mathbb{S}^{1}}\left\vert f\ast g^{\prime
}(t)\right\vert dt
\end{equation*}%
and by Lemma \ref{lemma1} 
\begin{equation*}
Var(f\ast g)\leq \Vert f||_{L^{1}}\Vert g^{\prime }||_{L^{1}}
\end{equation*}%
from which we get directly the statement.

Now suppose $f\in C^{1}$ and $g\in BV$, let us consider as before $%
g_{\varepsilon }\in C^{1}$ such that $\Vert g_{\varepsilon }-g||_{L^{1}}\leq
\varepsilon $ and $\Vert g_{\varepsilon }\Vert _{BV}=\Vert g\Vert _{BV}.$

Now 
\begin{eqnarray*}
Var(f\ast g) &\leq &Var(f\ast g-g_{\epsilon })+Var(f\ast g_{\epsilon }) \\
&\leq &Var(f\ast g-g_{\epsilon })+\Vert f||_{L^{1}}\Vert g_{\epsilon
}^{^{\prime }}||_{L^{1}}
\end{eqnarray*}%
but $Var(f\ast g-g_{\epsilon })\leq Var(f)\Vert g-g_{\epsilon }\Vert \leq
\epsilon Var(f)$ and can be made as small as wanted, allowing to prove the
statement in the case $f\in C^{1}$ and $g\in BV.$

Now by approximating $f\in L^{1}$ by a $f_{\varepsilon }\in C^{1}$ such that 
$\Vert f_{\varepsilon }-f||_{L^{1}}\leq \varepsilon $ and using again Lemma %
\ref{lemma1} we get the full statement.
\end{proof}

\subsubsection{Linear response, a general statement.}

In \cite{GG}, systems with additive noise are considered and a linear
response theorem is proved for a general class of Markov operators including
dynamical systems with additive noise. We state the theorem and apply it to
the random Arnold maps and perturbations we mean to consider. \newline
We will consider a normed vector space $(B_{w},\Vert \cdot \Vert _{w})$,
with $BS\supseteq B_{w}\supseteq L^{1}$ and $\Vert \cdot \Vert _{w}\leq
\Vert \cdot ||_{L^{1}}$, as well as its space of zero average measure $V_{w}$%
. 
\begin{equation}
V_{w}:=\{\mu \in B_{w}~s.t.~~\mu (S^{1})=0\}.  \label{d2}
\end{equation}%
Let us consider a family of Markov operators $L_{\delta
}:BS(S^{1})\rightarrow BS(S^{1})$, where $\delta \in \lbrack 0,\overline{%
\delta })$.

Recall that a Markov operator $L$ is positive and preserves probability
measures: if $f\geq 0$ then $Lf\geq 0$ and $f(S^{1})=Lf(S^{1})$ for each $%
f\in BS(S^{1})$. Let us denote by ${Id}$ the identity operator, and by $(z{Id%
}-L)^{-1}$ the resolvent related to an operator $L$, formally defined as%
\begin{equation}
(z{Id}-L)^{-1}=\sum_{n=0}^{\infty }\dfrac{1}{z^{n+1}}L^{n}.  \label{d3}
\end{equation}%
which is rigorously defined on suitable spaces whenever the infinite series
converges. Let us suppose that each operator $L_{\delta }$ has a fixed
probability measure in $BV(S^{1})$. We now show that under mild further
assumptions these fixed points vary smoothly in the weaker norm $\Vert \cdot
\Vert _{w}$. The following general result on linear response for \ system
with additive noise is proved in \cite{GG}.

\begin{theorem}
\label{th:linearresponse} Suppose that the family of operators $L_{\delta}$
satisfies the following {\ four conditions:}

\begin{itemize}
\item[(LR0)] $f_{\delta }\in BV(S^{1})$ is a probability measure such that $%
L_{\delta }f_{\delta }=f_{\delta }$ for each $\delta \in \left[ 0,\overline{%
\delta }\right) $. Moreover there is $M\geq 0$ such that $\Vert f_{\delta
}\Vert _{BV}\leq M$ for each $\delta \in \left[ 0,\overline{\delta }\right) $%
.

\item[(LR1)] (mixing for the unperturbed operator) For each $g\in BV(S^{1})$
with ${\int g\,dm=0}$ 
\begin{equation*}
\lim_{n\rightarrow \infty }\Vert L_{0}^{n}g||_{L^{1}}=0.
\end{equation*}

\item[(LR2)] (regularization of the unperturbed operator) $L_{0}$ is
regularizing from $B_{w}$ to $L^{1}$ and from $L^{1}$ to Bounded Variation
i.e. $L_{0}:(B_{w},\Vert \cdot \Vert _{w})\rightarrow L^{1}$ , ${\
L_{0}:L^{1}\rightarrow BV}$ are continuous.

\item[(LR3)] (small perturbation and derivative operator) There is $K\geq 0$
such that\footnote{\textbf{Notation:} If $A,B$ are two normed vector spaces
and $T:A\rightarrow B$ we write $\Vert T\Vert _{A\rightarrow B}:=\sup_{f\in
A,\Vert f\Vert _{A}\leq 1}\Vert Tf\Vert _{B}$} $\left\vert |L_{0}-L_{\delta
}|\right\vert _{L^{1}\rightarrow (B_{w},\Vert \cdot \Vert _{w})}\leq K\delta
,$ and $\left\vert |L_{0}-L_{\delta }|\right\vert _{BV\rightarrow V}\leq
K\delta $. There is ${\dot{L}f_{0}\in V_{w}}$ such that 
\begin{equation}
\underset{\delta \rightarrow 0}{\lim }\left\Vert \frac{(L_{0}-L_{\delta })}{
\delta }f_{0}-\dot{L}f_{0}\right\Vert _{w}=0.  \label{derivativeoperator}
\end{equation}
\end{itemize}

Then $({Id}-L_{0})^{-1}:V_{w}\rightarrow $ $V_{w}$ is a continuous operator
and we have the following linear response formula 
\begin{equation}
\lim_{\delta \rightarrow 0}\left\Vert \frac{f_{\delta }-f_{0}}{\delta }-({Id}%
-L_{0})^{-1}\dot{L}f_{0}\right\Vert _{w}=0.  \label{linresp}
\end{equation}%
Thus $({Id}-L_{0})^{-1}\dot{L}f_{0}$ represents the first-order term in the
change of equilibrium measure for the family of systems $L_{\delta }$.%
\newline
\end{theorem}

\begin{remark}
{Condition} LR0 is always satisfied by systems with additive noise.
Furthermore, the stationary measure $f_{0}$ has a density of bounded
variation; see \cite{GG}, Lemma 23.
\end{remark}

\begin{remark}
By Conditions LR1 and LR2, $f_{0}$ is the unique fixed probability measure
of $L_{0}$ in $BS(S^{1})$.
\end{remark}

\begin{remark}
The regularization property called for by Condition LR2 is required only for
the \emph{unperturbed} operator $L_{0}$. In our systems this regularization
is provided by the estimates shown in Section \ref{Reg}.
\end{remark}

\begin{remark}
The mixing assumptions {\ listed in Condition LR1 are} required only for the 
\emph{unperturbed} operator $L_{0}$. This {\ statement will be proved for
certain} classes of systems by a computer-aided proof.
\end{remark}

\begin{remark}
In Theorem \ref{th:linearresponse} the weak norm $\Vert \cdot \Vert _{w}$
could be the $L^{1}$ norm itself. In our case we can prove the existence of
the limit in $(\ref{derivativeoperator})$ with a convergence in the
Wasserstein $\Vert ~\Vert _{W}$ norm, this will be the choice of the weak
norm that will be used in this paper.
\end{remark}

\subsection{\ Linear response for Arnold maps with noise}

\label{ssec:LR_Arnold} {\ We will now apply} the above theorem to a family
of operators $L_{\delta },$ $\delta \in \lbrack 0,\overline{\delta }]$,
which are the annealed transfer operators associated with the Arnold maps
with additive noise defined in $($\ref{proc}$)$. {Recall that the transfer
operator associated to a deterministic transformation $T_{\delta }$ is
defined, as usual by the pushforward map also denoted by $(T_{\delta
})_{\ast },$ by 
\begin{equation}
\lbrack L_{T_{\delta }}(\mu )](A):=[(T_{\delta })_{\ast }\mu ](A):=\mu
(T_{\delta }^{-1}(A))  \label{pish}
\end{equation}%
for each measure with sign $\mu $ and Borel set $A$. When $T$ is nonsingular%
\footnote{%
A Borel map $T:X\rightarrow X$ is a said to be nonsingular with respect to
the Lebesgue measure $m$ if for any measurable $N$ $m(T^{-1}(N))=0\iff
m(N)=0 $ .} $L_{T}$ preserves the space of absolutely continuous finite
measures and can be considered as an operator $L^{1}\rightarrow L^{1}$. }The
transfer operator associated to the Arnold map with additive noise (see \cite%
{LM} or \cite{Viana} for basic facts on transfer operators associated to
random systems) will be the composition of the transfer operator $%
L_{T_{\delta }}$ related to the map 
\begin{equation*}
T_{\delta }:=T_{\tau +\delta ,\epsilon }(x)=x+\tau +\delta -\frac{\epsilon }{%
2\pi }\sin (2\pi x)
\end{equation*}%
and the action of te noise which is given by a convolution. The transfer
operator associated to the system with additive noise $L_{\delta
}:L^{1}\rightarrow L^{1}$ is then given by%
\begin{equation}
\lbrack L_{\delta }f](t)=[\rho _{\xi }\ast L_{T_{\delta }}(f)](t).
\label{1111}
\end{equation}

As said before, to apply Theorem \ref{th:linearresponse} to our case we
consider $\Vert \Vert _{w}$ as the $W$ norm. We need to prove that the
assumptions are satisfied. The most complicated one is the assumption $LR3$,
the existence of the derivative operator

\begin{lemma}
\label{012}The limit defined at $($\ref{derivativeoperator}$)$ exists in $BS$
and the limit converges in the W-norm.%
\begin{equation*}
\underset{\delta \rightarrow 0}{\lim }\left\Vert \frac{(L_{0}-L_{\delta })}{%
\delta }f_{0}-\frac{[\delta _{-\xi }-\delta _{\xi }]}{2\xi }\ast
L_{T_{0}}(f_{0})\right\Vert _{W}=0.
\end{equation*}
\end{lemma}

\begin{proof}
We remark that $L_{T_{0}}(f_{0})\in L^{1}(S^{1})$ furthermore%
\begin{equation*}
L_{0}f_{0}(t)=[\rho _{\xi }\ast L_{T_{0}}(f_{0})](t)=\int_{\mathbb{S}%
^{1}}\rho _{\xi }(t-\tau )[L_{T_{0}}(f_{0})](\tau )d\tau
\end{equation*}%
and

\begin{eqnarray*}
L_{\delta }f_{0}(t) &=&[\rho _{\xi }\ast L_{T_{\delta }}(f_{0})](t)=\int_{%
\mathbb{S}^{1}}\rho _{\xi }(t-\tau )[L_{T_{0}}(f_{0})](\tau -\delta )d\tau \\
&=&\int_{\mathbb{S}^{1}}\rho _{\xi }(t-c-\delta )[L_{T_{0}}(f_{0})](c)dc
\end{eqnarray*}

hence%
\begin{equation}
\left[ \frac{L_{\delta }-L_{0}}{\delta }f_{0}\right] (x)=\int_{\mathbb{S}%
^{1}}\frac{\rho _{\xi }(t-\delta -\tau )-\rho _{\xi }(t-\tau )}{\delta }%
[L_{T_{0}}(f_{0})](\tau )d\tau .  \label{key}
\end{equation}

Denoting by $R_{\delta }:L^{1}\rightarrow L^{1}$ the translation operator
given by $[R_{\delta }f](x)=f(x-\delta )$. It holds that%
\begin{equation}
\frac{L_{\delta }-L_{0}}{\delta }f_{0}=\frac{[R_{-\delta }\rho _{\xi }-\rho
_{\xi }]}{\delta }\ast \lbrack L_{T_{0}}(f_{0})]  \label{(key2)}
\end{equation}%
but 
\begin{equation*}
\lim_{\delta \rightarrow 0}\left\Vert \frac{[R_{-\delta }\rho _{\xi }-\rho
_{\xi }]}{\delta }-\frac{[\delta _{-\xi }-\delta _{\xi }]}{2\xi }\right\Vert
_{W}=0
\end{equation*}%
where $\delta _{-\xi }$ and $\delta _{\xi }$ are the delta measures placed
on $\pm \xi $.

Then by (\ref{conv1}) 
\begin{equation*}
\lim_{\delta \rightarrow 0}\frac{L_{0}-L_{\delta }}{\delta }f_{0}=\frac{%
[\delta _{-\xi }-\delta _{\xi }]}{2\xi }\ast L_{T_{0}}(f_{0})
\end{equation*}%
(recall that $T_{0}$ is nonsingular and then $L_{T_{0}}(f_{0})\in L^{1}$ )
with convergence in the $||~||_{W}$ norm.
\end{proof}

\begin{lemma}
\label{123}The remaining assumptions of Item $LR3$ of Theorem \ref%
{th:linearresponse} are satisfied: 
\begin{equation*}
\left\vert |L_{0}-L_{\delta }|\right\vert _{L^{1}\rightarrow (BS,\Vert \cdot
\Vert _{W})}\leq K\delta ,~and~\left\vert |L_{0}-L_{\delta }|\right\vert
_{BV\rightarrow V}\leq K\delta .
\end{equation*}
\end{lemma}

\begin{proof}
Let $f\in L^{1}$. By (\ref{(key2)}) it holds%
\begin{equation*}
\lbrack L_{\delta }-L_{0}]f=[R_{-\delta }\rho _{\xi }-\rho _{\xi }]\ast
\lbrack L_{T_{0}}(f)].
\end{equation*}

Since there is a $K$ such that $||[R_{-\delta }\rho _{\xi }-\rho _{\xi
}]||_{L^{1}}\leq K\delta $ \ by Lemmas \ref{lemma1} and \ref{convoo copy(1)}
\ we directly get the statement.
\end{proof}

All the estimates of this section lead to the following linear response
statement for the systems and perturbations we consider in this work.

\begin{proposition}
\label{proplinresp}Let $T_{0}:S^{1}\rightarrow S^{1}$ be a nonsingular map.
Let $T_{\delta }$ defined as $T_{\delta }(x)=T_{0}(x)+\delta $, let $%
L_{\delta }:L^{1}\rightarrow L^{1}$ be the transfer operator defined as in $(%
\ref{1111})$. Let $f_{\delta }\in L^{1}$ be such that $L_{\delta }f_{\delta
}=f_{\delta }$ (a stationary measure for the system $L_{\delta }$).

Suppose $L_{0}$ is mixing: for every $g\in BV[0,1]$ with $\int_{I}g\,dm=0$,
then%
\begin{equation*}
\lim_{n\rightarrow \infty }\Vert L_{0}^{n}g||_{L^{1}}=0.
\end{equation*}%
(see Assumption LR1 of Theorem \ref{th:linearresponse}) Then $({Id}%
-L_{0})^{-1}$ is a continuous operator on the space of zero average Borel
measures equipped with the $||~||_{W}$ norm and 
\begin{equation}
\lim_{\delta \rightarrow 0}\left\Vert \frac{f_{\delta }-f_{0}}{\delta }-({Id}%
-L_{0})^{-1}[\frac{[\delta _{-\xi }-\delta _{\xi }]}{2\xi }\ast
L_{T_{0}}(f_{0})]\right\Vert _{W}=0.
\end{equation}
\end{proposition}

\begin{proof}
The proof is a direct application of Theorem \ref{th:linearresponse}. The
assumption $LR2$ of Theorem \ref{th:linearresponse} is a direct consequence
of Lemma \ref{lemmconv2} and Lemma \ref{lemmaconv4}. The assumption $LR3$
and the formula for the derivative operator is proved in Lemmas \ref{012} \
and \ref{123}.
\end{proof}

The linear response in $\Vert ~\Vert _{W}$ is sufficient to deduce the
smoothness of the rotation number because the observable associated to the
rotation number is Lipschitz. As already done in \cite{ZH07} we exploit the
fact that the rotation number can be computed as the integral of a suitable
observable with respect to the stationary measure. \newline
To formulate this precisely, we introduce a few notations:

\begin{enumerate}
\item Denoting $[-\xi/2,\xi/2]^{\mathbb{N}}=\Omega$, we introduce the
one-sided shift $\sigma:\Omega\to\Omega$, classically defined for $%
\boldsymbol{\omega}=(\omega_n)_{n\geq 0}$ by $\sigma(\boldsymbol{\omega}%
)=(\omega_n)_{n\geq 1}$. \newline
We let $\nu=\dfrac{1}{\xi}Leb_{[-\xi/2,\xi/2]}$, and let $\mathbb{P}$ be the
product measure on $\Omega$. It is an invariant probability measure for $%
\sigma$.

\item Let $\phi _{\tau }:\Omega \times \mathbb{R}\rightarrow \mathbb{R}$ be
the map defined by 
\begin{equation*}
\phi _{\tau }(\boldsymbol{\omega },x)=\phi _{\tau }(\omega _{0},x)=\tau
+\omega _{0}-\dfrac{\epsilon }{2\pi }\sin (2x\pi ).
\end{equation*}%
Notice that it is 1-periodic in $x$: thus it induces an observable of the
circle $\mathbb{S}^{1}$. Furthermore, it is the lift of $X_{1}-Id$ to $%
\mathbb{R}$. \newline
We also let $\varphi _{\tau }$ be the lift of $T_{\tau ,\epsilon }-Id$ to $%
\mathbb{R}$. For the same reasons, it induces an observable on $\mathbb{S}%
^{1}$.

\item Finally, we let $F_{\tau }:\Omega \times \mathbb{S}^{1}\rightarrow
\Omega \times \mathbb{S}^{1}$ be the skew-product map 
\begin{equation*}
F_{\tau }(\boldsymbol{\omega },x)=(\sigma (\boldsymbol{\omega }),T_{\tau
,\epsilon }(x)+\omega _{0})=(\sigma (\boldsymbol{\omega }),X_{1}(\omega
_{0},x))
\end{equation*}%
for $\boldsymbol{\omega }=(\omega _{n})_{n\geq 0}\in \Omega $. The product
measure $\mathbb{P}\otimes \mu _{\tau }$ (where $\mu _{\tau }$ is the
stationary measure of the Arnold map $T_{\tau ,\epsilon }$ with uniformly
distributed additive noise of size $\xi $) is invariant, and in the case
where the system satisfies the mixing assumption LR1 is also ergodic for $%
F_{\tau }$ (see \cite{Viana}, Section 5) by this we can now formulate:
\end{enumerate}

\begin{proposition}
\label{integral} Let $T_{\tau ,\epsilon }$ be the {\ Arnold map} with
parameters $(\tau ,\epsilon )$ and uniformly distributed noise of size $\xi $%
, suppose the system satisfies the assumption $LR1$, let $\mu _{\tau }$ be
the corresponding stationary measure and $\rho _{\tau }$ be the associated
rotation number. Then 
\begin{equation}
\rho _{\tau }(\boldsymbol{\omega },x)=\int_{\mathbb{S}^{1}}\varphi _{\tau
}d\mu _{\tau }.  \label{eq:rotnumbasLipobservable}
\end{equation}%
In particular, $\rho _{\tau }$ is $\mathbb{P}\otimes \mu _{\tau }$ almost
surely constant.
\end{proposition}

\begin{proof}
With the notation of \eqref{proc} one can write the $n$-th iterate of
the skew-product map as 
\begin{equation*}
F^{n}(\boldsymbol{\omega },x)=(\sigma ^{n}(\boldsymbol{\omega }%
),X_{n}(\omega _{n-1},\dots ,\omega _{0},x)).
\end{equation*}%
Considering the Birkhoff sum associated to this system and the observable $%
\phi _{\tau }$, one has: 
\begin{equation*}
\dfrac{1}{N}\sum_{n=0}^{N-1}\phi _{\tau }\circ F^{n}=\dfrac{\hat{X}_{N}-Id}{N%
},
\end{equation*}%
where $\hat{X}_{N}$ as before is the lift of $X_{N}$ to $\mathbb{R}$. Here
we commit a slight abuse of notation, as $x\in \mathbb{S}^{1}$. Note however
that this abuse is justified by the fact that $\hat{X}_{N}-Id$ is a
one-periodic map. This Birkhoff sum is the lift of $\dfrac{1}{N}(X_{N}-Id)$:
by definition of the rotation number, this right-hand side converges, as $%
N\rightarrow \infty $, to $\rho _{\tau }(x,\boldsymbol{\omega })$. \newline
But by Birkhoff theorem, the left-hand side $\phi _{\tau }$ converges to 
\begin{equation*}
\int_{\mathbb{S}^{1}}\int_{\Omega }\phi _{\tau }(\boldsymbol{\omega },x)d%
\mathbb{P}(\boldsymbol{\omega })d\mu _{\tau }(x)=\int_{\mathbb{S}%
^{1}}\int_{[-\xi /2,\xi /2]}\phi _{\tau }(\omega _{0},x)d\nu (\omega
_{0})d\mu _{\tau }(x).
\end{equation*}%
Now, it is easy to see that for fixed $x\in \mathbb{S}^{1}$, one has: 
\begin{equation*}
\int_{\lbrack -\xi /2,\xi /2]}\phi _{\tau }(\omega _{0},x)d\nu =\tau -\dfrac{%
\epsilon }{2\pi }\sin (2\pi x)=\varphi _{\tau }(x).
\end{equation*}%
Thus one obtains \eqref{eq:rotnumbasLipobservable}, as announced.

\end{proof}

\begin{corollary}
\label{corfin} The rotation number of the {\ Arnold maps} with uniformly
distributed additive noise is differentiable at each value of the parameter $%
\tau $ for which the associated system is mixing (in the sense stated in
assumption LR1 and Proposition \ref{proplinresp}). Furthermore if $\tau _{0}$
is such a parameter we get the following formula for the derivative of the
rotation number computed at $\tau _{0}$%
\begin{equation}
\lbrack \frac{d}{d\tau }\rho _{\tau }](\tau _{0})=1+\int_{\mathbb{S}%
^{1}}\varphi _{\tau _{0}}~d[({Id}-L_{\tau _{0}})^{-1}\frac{[\delta _{-\xi
}-\delta _{\xi }]}{2\xi }\ast L_{T_{\tau _{0}}}(\mu _{\tau _{0}})]
\label{formfin}
\end{equation}%
where $L_{T_{\tau _{0}},\epsilon }$ is the pushforward operator of the map $%
T_{\tau _{0},\epsilon }$.
\end{corollary}

\begin{proof}
By Proposition \ref{integral} the rotation number is the integral of a
Lipschitz observable. Considering the increment of $\rho _{\tau }$we get%
\begin{eqnarray*}
\frac{\rho _{\tau }(\tau _{0}-h)-\rho _{\tau }(\tau _{0})}{h} &=&\frac{1}{h}%
\left[ \int_{\mathbb{S}^{1}}\varphi _{\tau _{0}+h}~d\mu _{\tau _{0}+h}-\int_{%
\mathbb{S}^{1}}\varphi _{\tau _{0}}d\mu _{\tau _{0}}\right] \\
&=&\frac{1}{h}\left[ \int_{\mathbb{S}^{1}}\varphi _{\tau _{0}+h}~d\mu _{\tau
_{0}+h}-\int_{\mathbb{S}^{1}}\varphi _{\tau _{0}}d\mu _{\tau _{0+h}}\right]
\\
&&+\frac{1}{h}\left[ \int_{\mathbb{S}^{1}}\varphi _{\tau _{0}}~d\mu _{\tau
_{0}+h}-\int_{\mathbb{S}^{1}}\varphi _{\tau _{0}}d\mu _{\tau _{0}}\right] .
\end{eqnarray*}%
Here%
\begin{equation*}
\frac{1}{h}\left[ \int_{\mathbb{S}^{1}}\varphi _{\tau _{0}+h}~d\mu _{\tau
_{0}+h}-\int_{\mathbb{S}^{1}}\varphi _{\tau _{0}}d\mu _{\tau _{0+h}}\right]
=\int_{\mathbb{S}^{1}}\frac{\varphi _{\tau _{0}+h}-\varphi _{\tau _{0}}}{h}%
d\mu _{\tau _{0+h}}=1
\end{equation*}%
and%
\begin{equation*}
\frac{1}{h}\left[ \int_{\mathbb{S}^{1}}\varphi _{\tau _{0}}~d\mu _{\tau
_{0}+h}-\int_{\mathbb{S}^{1}}\varphi _{\tau _{0}}d\mu _{\tau _{0}}\right]
=\int_{\mathbb{S}^{1}}\varphi _{\tau _{0}}d[\frac{\mu _{\tau _{0}+h}-\mu
_{\tau _{0}}}{h}]
\end{equation*}%
and the statement directly follows from Proposition \ref{proplinresp}.
\end{proof}

In next section we will show explicit examples of cases in which the system
is mixing and $($\ref{formfin}$)$ holds.

\section{{\ Mixing rate properties}\label{comptool}}

In this section we show families of Arnold maps with uniformly distributed
additive noise which are mixing systems in the sense of Assumption $LR1$ of
Theorem \ref{th:linearresponse}. By applying Corollary \ref{corfin} we then
get differentiability of the rotation number in these sets of examples. The
method used in the verification of the mixing assumption is based on a
computer aided estimate and a further "stability of mixing" estimate.

In \cite{GMN}, Section 4 it is shown how to use a computer aided estimate to
prove that a given system with additive noise is mixing. The algorithm is
based on the approximation of the system's transfer operator with a finite
rank operator (a finite element approach). In this approximation strategy it
is possible to get explicit bounds to the various approximation errors. The
system is then approximated by a finite Markov chain whose behavior can be
rigorously investigated by the computer (again with rigorous bounds on the
numerical errors provided by a suitable implementation using interval
arithmetics). Putting together the information coming from the certified
estimates done by the computer and the explicit functional analytic\
estimates on the approximation we can extract information on the behavior of
the original system. Given a system made by a map of the interval with
additive noise of range $\xi $ and associated transfer operator $L$, the
algorithm\footnote{%
The algorithm and the code used in this work (see Note \ref{notacode}) is
almost identical to the one used in \cite{GMN}. The only important
difference is the fact that in our code the convolution on $\mathbb{S}^{1}$
is implemented, while in the original work of \cite{GMN} a reflecting
boundaries convolution is considered.} can certify an $\alpha <1$ and $n\in 
\mathbb{N}$ such that $||L^{n}||_{V\rightarrow L^{1}}\leq \alpha $ implying
exponential contraction of the zero average space.

Since we want to obtain that this assumption is satisfied in some large set
of examples, we have to perform a slightly more complicated construction. We
first show in Subsection \ref{allarga} that if a system with noise is mixing
then also an open set of nearby systems are mixing, this is done also giving
an explicit estimate for the radius of this open neighborhood. Then, in
Subsection \ref{exp1} we apply the computer aided estimates to a certain
suitable finite family of systems (a kind of $\epsilon $-grid in the space
of parameters). We then get that these systems satisfy $LR1$ with related
neighborhoods covering a large set in the parameters spaces. Putting these
two steps together we hence have a large set of parameters on which $LR1$
applies.

\subsection{Rate of mixing and perturbations\label{allarga}}

Suppose a given system with additive noise is proved to be mixing. In this
section we provide the theoretical tools to extend the mixing to nearby
systems, showing that mixing is indeed stable when the system is suitably
perturbed. We provide quantitative estimates on this stability. Another
application of these estimates is to provide mixing and mixing rate of a
system when the noise distribution is changed. For example getting mixing
rate for the gaussian noise once the mixing rate for a suitable uniform
noise is established.

\subsubsection{Perturbing the map}

In this subsection we start considering perturbations of the map. We compare
the mixing rate of a given system with the one of a system where the
deterministic part of the dynamics gets a small perturbation in $%
||~||_{\infty }$ norm.

\begin{definition}
A piecewise continuous map $T$ on $[0,1]$ is a function $T:[0,1]\rightarrow
\lbrack 0,1]$ such that there is partition $\{I_{i}\}_{1\leq i\leq k}$ of $%
[0,1]$ made of intervals $I_{i}$ such that $T$ has a continuous extension to
the closure $\bar{I}_{i}$ of each interval.
\end{definition}

\begin{definition}
(notations) Let us define the convolution operator $N:L^{1}\rightarrow BV$
defined by%
\begin{equation}
Nf=\rho \ast f  \label{NNN}
\end{equation}%
and $L_{T_{1}},L_{T_{2}}$ be the the transfer operators associated to the
maps $T_{1},T_{2}$. Let $V$ be the set of zero average measures in $L^{1}$
as defined in $($\ref{d1}$).$
\end{definition}

In this framework we prove the following

\begin{proposition}
\label{111}Let $T_{1}$ and $T_{2}:[0,1]\rightarrow \lbrack 0,1]$ be
piecewise continuous nonsingular maps and $\rho \in BV$. With the notations
introduced above, for any $f\in L^{1}$it holds%
\begin{equation*}
||(NL_{T_{1}})^{n}f-(NL_{T_{2}})^{n}f||_{L^{1}}\leq 2n\Vert T_{1}-T_{2}\Vert
_{L^{\infty }}\cdot \Vert \rho \Vert _{BV}\cdot \Vert f\Vert _{{L}^{1}}.
\end{equation*}
\end{proposition}

Before the proof of Proposition \ref{111} we need some preliminary lemmas.

\begin{lemma}
Let $T_{1}$ and $T_{2}:[0,1]\rightarrow \lbrack 0,1]$ be piecewise
continuous nonsingular maps and let $L_{1},L_{2}$ the associated
deterministic transfer operators, let $f\in L^{1}$. Then%
\begin{equation*}
\Vert L_{T_{1}}(f)-L_{T_{2}}(f)\Vert _{W}\leq \Vert T_{1}-T_{2}\Vert
_{L^{\infty }}\cdot \Vert f||_{L^{1}}.
\end{equation*}
\end{lemma}

\begin{proof}
The proof of the statement is straightforward, using the uniform continuity
of each branch of $T_{1}$ and $T_{2}$ to suitably approximate $f$ with a
combination of delta measures.
\end{proof}

\begin{lemma}
\label{prt copy(1)}Let $T_{1}$ and $T_{2}:[0,1]\rightarrow \lbrack 0,1]$ be
piecewise continuous nonsingular maps and $\rho \in BV.$ Let the associated
transfer operators with additive noise diven by the kernel $\rho $ be $%
NL_{T_{1}}$, $NL_{T_{2}}$, then for any $f\in L^{1}$ it holds%
\begin{equation*}
\Vert NL_{T_{1}}(f)-NL_{T_{2}}(f)||_{L^{1}}\leq 2\Vert T_{1}-T_{2}\Vert
_{L^{\infty }}\cdot \Vert \rho ||_{BV}\cdot \Vert f||_{L^{1}}.
\end{equation*}
\end{lemma}

\begin{proof}
Indeed, 
\begin{eqnarray*}
\Vert NL_{T_{1}}(f)-NL_{T_{2}}(f)||_{L^{1}} &=&\Vert \rho \ast
(NL_{T_{1}}(f)-NL_{T_{2}}(f))||_{L^{1}} \\
&\leq &2\Vert \rho \Vert _{BV}\cdot \Vert L_{T_{1}}(f)-L_{T_{2}}(f)\Vert
_{W}.
\end{eqnarray*}%
Using Lemma \ref{lemmconv2}\ in the last estimate, which gives the statement.
\end{proof}

\begin{proof}[Proof of Proposition \protect\ref{111}]
We have%
\begin{eqnarray*}
(NL_{T_{1}})^{n}-(NL_{T_{2}})^{n}
&=&\sum_{k=1}^{n}(NL_{T_{1}})^{n-k}(NL_{T_{1}}-NL_{T_{2}})(NL_{T_{2}})^{k-1}
\\
||(NL_{T_{1}})^{n}-(NL_{T_{2}})^{n}||_{L^{1}\rightarrow L^{1}} &\leq
&n||NL_{T_{1}}-NL_{T_{2}}||_{L^{1}\rightarrow L^{1}}.
\end{eqnarray*}%
Estimating $||NL_{T_{1}}-NL_{T_{2}}||_{L^{1}\rightarrow L^{1}}$ by Lemma \ref%
{prt copy(1)} we get the statement.
\end{proof}

The following corollary directly follow from Proposition \ref{prt copy(1)}
and show how to use it to estimate the rate of mixing of a perturbation

\begin{corollary}
\label{ext1} If%
\begin{equation}
||(NL_{T_{1}})^{n}f||_{V\rightarrow L^{1}}\leq \alpha <1  \label{lei}
\end{equation}%
then 
\begin{equation*}
||(NL_{T_{2}})^{n}f||_{V\rightarrow L^{1}}\leq \alpha +2n\Vert
T_{1}-T_{2}\Vert _{L^{\infty }}\cdot \Vert \rho \Vert _{BV}.
\end{equation*}
\end{corollary}

Corollary \ref{ext1} will be used in the following way: suppose to have
proved the mixing for the operator $NL_{T_{1}}$, i.e. we computed $n,\alpha $
for which \ref{lei} is satisfied, then Corollary \ref{ext1} implies that all
the operators $T_2$ such that 
\begin{equation}  \label{eq_ext}
\Vert T_{1}-T_{2}\Vert _{L^{\infty }}<\frac{1-\alpha }{2n\Vert \rho \Vert
_{BV}}
\end{equation}%
are still mixing.

\subsubsection{Perturbing the noise}

In this subsection we change the noise kernel with a small perturbation in $%
L^{1}.$ We see that the iterates of the new system are still near to the
ones of the original system, and thus we can estimate the rate of mixing.

\noindent \textbf{Notation} Let $\rho _{1,}$ $\rho _{2}\in BV$ be two noise
kernels let us denote by $N_{1},N_{2}$ the associated convolution operators
as defined in $($\ref{NNN}$)$, {then it holds}

\begin{proposition}
\label{222}For each $n\in \mathbb{N}$%
\begin{equation*}
||(N_{1}L)^{n}f-(N_{2}L)^{n}f||_{L^{1}}\leq n||\rho _{1}-\rho
_{2}||_{L^{1}}||f||_{L^{1}}.
\end{equation*}
\end{proposition}

\begin{proof}
The proof is immediate%
\begin{eqnarray*}
(N_{1}L)^{n}-(N_{2}L)^{n}
&=&\sum_{k=1}^{n}(N_{1}L)^{n-k}(N_{1}L-N_{2}L)(N_{2}L)^{k-1} \\
||(N_{1}L)^{n}-(N_{2}L)^{n}||_{L^{1}} &\leq
&n||N_{1}-N_{2}||_{L^{1}\rightarrow L^{1}}\leq n||\rho _{1}-\rho
_{2}||_{L^{1}}.
\end{eqnarray*}
\end{proof}

Next corollary directly follow from Proposition \ref{222} and show to
estimate the rate of mixing of a perturbation of the operator.

\begin{corollary}
If%
\begin{equation*}
||(N_{1}L)^{n}f||_{V\rightarrow L^{1}}\leq \alpha <1
\end{equation*}%
then 
\begin{equation*}
||(N_{2}L)^{n}f||_{V\rightarrow L^{1}}\leq \alpha +n||\rho _{1}-\rho
_{2}||_{L^{1}}.
\end{equation*}
\end{corollary}

\subsection{{Computer-aided estimates on the mixing rate}.\label{exp1}}

In this subsection we show the results of the computer aided estimates for a
family of systems with strong nonlinearity, $\epsilon =1.4$, with noise of
magnitude $\xi =0.1$ (similarly to the noise range considered in \cite{GCS08}%
). For several values of $\tau $. As explained at beginning of Section \ref%
{comptool} we use the algorithm described in Section 4 of \cite{GMN} to
prove that assumption $LR1$ holds for these systems and find $n$ such that 
\begin{equation}
||L_{\xi }^{n}||_{V\rightarrow L^{1}}\leq \alpha <1.  \label{leilei}
\end{equation}%
Then we use the theory developed in Subsection \ref{allarga} to provide the
mixing rate for nearby systems, obtaining a large interval.

\begin{proposition}
Let $\epsilon =1.4$ and $\xi =0.1$. Then, for each $\tau \in \lbrack
0.75,0.8]$\footnote{%
The zip file Arnold.results.zip at %
\url{https://bitbucket.org/luigimarangio/arnold_map/} contains the results
of more than 300 computer-aided estimates, including the ones listed in
table 2; these numerical data extend the result of proposition 27 for $\tau
\in \lbrack 0.7,0.8]$.}, the corresponding Arnold map with noise and
parameters $(\epsilon ,\tau ,\xi )$ satisfy assumption $LR1$ of Theorem \ref%
{th:linearresponse}.
\end{proposition}

\begin{proof}
Suppose we have proved the mixing for a system with parameters $(\epsilon
_{0},\tau _{0},\xi _{0})$ and we have $\alpha $ and $n$ such that $($\ref%
{leilei}$)$ is satisfied, then $($\ref{eq_ext}$)$ implies that there exists $%
\theta _{0}>0$ and a whole interval $I_{0}=[\tau _{0}-\theta _{0},\tau
_{0}+\theta _{0}]$, such that all the systems with parameters $(\epsilon
_{0},\tau ,\xi _{0})$, $\tau \in I_{0}$, are still mixing. Moreover $($\ref%
{eq_ext}$)$ gives an explicit formula for $\theta _{0}$ which depends from $%
\alpha $ and $n$. These constants are explicitly computable with the
algorithm shown in \cite{GMN} which we are using in our code, hence we can
explicitly compute $\theta _{0}$.

To show that the mixing property holds for every system of parameters $%
(1.4,\tau ,0.1)$, with $\tau \in \lbrack 0.75,0.8]$, a strategy is to
consider a finite sequence of points $\{\tau _{i}\}\subset \lbrack 0.75,0.8]$
such that the systems with parameters $(1.4,\tau _{i},0.1)$ are mixing, for
every $i$, and the associated intervals $I_{i}:=[\tau _{i}-\theta _{i},\tau
_{i}+\theta _{i}]$, defined as above, cover the interval $[0.75,0.8]$.

In Table \ref{extmix} we show the computer aided estimates about the values
of $\theta _{0}$ by the method described above for each example. As it can
be seen, since the union of all this computed intervals is equal to $(a,b)$,
with $a=0.749399418088000$ and $b=0.800715949198087$, we have then proved
the desired property in the whole interval $[0.75,0.8]$.
\end{proof}

\begin{table}[tbp]
\caption{ Given the Arnold map with noise of magnitude $\protect\xi $ and
parameters $(\protect\tau _{0},\protect\epsilon _{0})$, for which we have
already proved mixing, the table shows the computed intervals $I_{0}=[%
\protect\tau _{0}-\protect\theta _{0},\protect\tau _{0}+\protect\theta _{0}]$%
, such that if $\protect\tau \in I_{0}$ then the Arnold map with parameters $%
(\protect\tau ,\protect\epsilon _{0})$ is mixing}
\label{extmix}\centering
{\ 
\begin{tabular}{ccc}
\toprule ${(\tau_0,\epsilon_0)}$ & ${\xi}$ & $[\tau_0-\theta_0,\tau_0+%
\theta_0]$ \\ 
\midrule $(0.7502,1.4)$ & $0.1$ & $[0.749399418088000, 0.751000581912001]$
\\ 
$(0.7516,1.4)$ & $0.1$ & $[0.750823157412727, 0.752376842587275]$ \\ 
$(0.7532,1.4)$ & $0.1$ & $[0.752340409081139, 0.754059590918862]$ \\ 
$(0.7548,1.4)$ & $0.1$ & $[0.753967757739070, 0.755632242260931]$ \\ 
$(0.7564,1.4)$ & $0.1$ & $[0.755472331622558, 0.757327668377443]$ \\ 
$(0.7582,1.4)$ & $0.1$ & $[0.757304135181951, 0.759095864818050]$ \\ 
$(0.7598,1.4)$ & $0.1$ & $[0.758934791149505, 0.760665208850496]$ \\ 
$(0.7616,1.4)$ & $0.1$ & $[0.760590252726379, 0.762609747273622]$ \\ 
$(0.7632,1.4)$ & $0.1$ & $[0.762227848809647, 0.764172151190354]$ \\ 
$(0.7648,1.4)$ & $0.1$ & $[0.763863841059881, 0.765736158940120]$ \\ 
$(0.7668,1.4)$ & $0.1$ & $[0.765692173699252, 0.767907826300749]$ \\ 
$(0.7688,1.4)$ & $0.1$ & $[0.767737148314721, 0.769862851685280]$ \\ 
$(0.7708,1.4)$ & $0.1$ & $[0.769780384398816, 0.771819615601185]$ \\ 
$(0.7728,1.4)$ & $0.1$ & $[0.771572710558602, 0.774027289441399]$ \\ 
$(0.7748,1.4)$ & $0.1$ & $[0.773628600754358, 0.775971399245643]$ \\ 
$(0.7768,1.4)$ & $0.1$ & $[0.775681672076595, 0.777918327923406]$ \\ 
$(0.7788,1.4)$ & $0.1$ & $[0.777626636524504, 0.779973363475497]$ \\ 
$(0.7812,1.4)$ & $0.1$ & $[0.779896782767258, 0.782503217232743]$ \\ 
$(0.7836,1.4)$ & $0.1$ & $[0.782295331693974, 0.784904668306028]$ \\ 
$(0.7860,1.4)$ & $0.1$ & $[0.784762771410257, 0.787237228589744]$ \\ 
$(0.7884,1.4)$ & $0.1$ & $[0.787166089347180, 0.789633910652821]$ \\ 
$(0.7908,1.4)$ & $0.1$ & $[0.789629611581057, 0.791970388418944]$ \\ 
$(0.7928,1.4)$ & $0.1$ & $[0.791684956386327, 0.793915043613674]$ \\ 
$(0.7948,1.4)$ & $0.1$ & $[0.793685568440505, 0.795914431559496]$ \\ 
$(0.7968,1.4)$ & $0.1$ & $[0.795685318203188, 0.797914681796813]$ \\ 
$(0.7988,1.4)$ & $0.1$ & $[0.797684476459441, 0.799915523540560]$ \\ 
$(0.7996,1.4)$ & $0.1$ & $[0.798484050801914, 0.800715949198087]$ \\ 
\bottomrule &  & 
\end{tabular}%
}
\end{table}

Once we have Assumption LR1 of Theorem \ref{th:linearresponse} satisfied for
this family of systems, applying Corollary \ref{corfin} we directly get

\begin{corollary}
Let $\epsilon =1.4$, $\xi =0.1$ then, for each $\tau \in \lbrack 0.75,0.8]$
the rotation number corresponding to the {\ Arnold map} with noise and
parameters $(\epsilon ,\tau ,\xi )$ is differentiable as $\tau $ varies and $%
($\ref{formfin}$)$ \ holds.
\end{corollary}


\section{{Non-monotonic rotation number {for} strong nonlinearity} \label%
{exp2}}

For $\epsilon <1$, in the case which the Arnold map is a diffeomorphism it
is well known that if $\tau _{2}>\tau _{1}$ then $\rho _{\tau _{2}}>\rho
_{\tau _{1}}$ (see \cite{wiggins2003}). In this section we show that the
rotation number is not necessarily monotonic anymore when $\epsilon >1$. We
prove this for a particular example with $\epsilon =1.4$ and $\xi =0.01.$
However non rigorous numerical experiments suggests that the phenomenon is
quite common when the noise is small (see Section \ref{last}). \ The proof
is done by rigorously approximating the value of the rotation number for
several values of $\tau $. This is done by the rigorous approximation of the
stationary measure with a small error in the $L^{1}$ norm using again the
algorithm described in \cite{GMN}. The algorithm approximates the transfer
operator by a finite rank one and approximates the stationary measure of the
initial transfer operator by the one of the approximated one. The algorithm
then validates this approximation, providing an explicit estimate on the
distance betweeen the two stationary measures in the $L^{1}$ norm by a
quantitative stability statement for the fixed points, not much different
from the one shown at Theorem \ref{th:linearresponse} (see \cite{GMN},
Section 3 for more details about how the algorithm works). This is
sufficient to get a certified estimate on the rotation number, as we have
seen, the rotation number is computable as the average of a Lipschitz
observable with respect to the stationary measure. More precisely,\ we will
find two values $\tau _{1}<\tau _{2}$ with corresponding rotation numbers $%
\rho _{1}\in I_{1}$, $\rho _{2}\in I_{2}$, where $I_{1}$ and $I_{2}$ are the
rigorous computed intervals in which the rotation numbers lie, furthermore,
the intervals are such that $max(I_{2})<min(I_{1})$. By this $\rho _{2}$
must be smaller then $\rho _{1}.$ Next proposition show an example of an
interval where we find a non monotonic behavior of the rotation number.

\begin{proposition}
\label{nonmonopr} Let $\epsilon = 1.4$ and $\xi = 0.01$, then the rotation
number $\rho_\tau$, as function of the parameter $\tau$, is not monotonic in
the interval $[0.707,0.716]$.
\end{proposition}

\begin{proof}
As explained above, we use the algorithm of \cite{GMN} to estimate of the
stationary measure for $\epsilon =1.4$ and for each $\tau \in \left\{
0.707,0.708,....,0.716\right\} $. We estimate the expected value of the
observable $\tilde{T}_{\tau ,\epsilon }(x)-x$ \ with respect to the
stationary measure for each example. This gives a certified interval in
which the rotation number $\rho _{\tau }$ of each example lies (see
Proposition \ref{integral}). The results are reported in Table \ref%
{nonmonotab}. The inspection of these, disjoint, decreasing, intervals shows
that the rotation number decreases for $\tau \in \left\{
0.707,0.708,....0.715\right\} $. The last estimate at $(0.716,1.4)$ shows an
increasing behavior, showing non monotonicity.
\end{proof}

Of course in the estimates shown in Table \ref{nonmonotab} \ and in the
previous proof, the "decreasing part" is more interesting, as this is
somewhat unexpected for an increasing forcing.

\begin{table}[tbp]
\caption{This table shows the computed intervals in which the rotation
number lies for each value of the parameters in consideration.}
\label{nonmonotab}\centering
{\ 
\begin{tabular}{ccc}
\toprule ${(\tau,\epsilon)}$ & ${\xi}$ & ${\rho_{\tau}}$ \\ 
\midrule $(0.707,1.4)$ & $0.01$ & $[0.780594, 0.780604]$ \\ 
$(0.708,1.4)$ & $0.01$ & $[0.778348, 0.778361]$ \\ 
$(0.709,1.4)$ & $0.01$ & $[0.775291, 0.775302]$ \\ 
$(0.710,1.4)$ & $0.01$ & $[0.771833, 0.771844]$ \\ 
$(0.711,1.4)$ & $0.01$ & $[0.768335, 0.768348]$ \\ 
$(0.712,1.4)$ & $0.01$ & $[0.765170, 0.765183]$ \\ 
$(0.713,1.4)$ & $0.01$ & $[0.762568, 0.762590]$ \\ 
$(0.714,1.4)$ & $0.01$ & $[0.760585, 0.760612]$ \\ 
$(0.715,1.4)$ & $0.01$ & $[0.759288, 0.759344]$ \\ 
$(0.716,1.4)$ & $0.01$ & $[0.759915, 0.759970]$ \\ 
\bottomrule &  & 
\end{tabular}%
}
\end{table}

\section{Comparison of the results with further numerical estimates\label%
{sec:Ulam}}

In this section we compare the approximation for the invariant measures
obtained with our certified method with two Monte Carlo methods, one based
on a Ulam method in which the estimates needed to set up a Markov
approximation of the system are made in a Monte Carlo way, and the other is
a pure Monte Carlo method, where we iterate long orbits. The Monte Carlo
numerical methods used in this section are then non-rigorous and they do not
provide certified bounds on the accuracy of the estimates, however they
confirm the previous findings and complements them with further details. We
now provide a short description of both Monte Carlo methods used in this
section.

\subsection{Ulam's Monte Carlo method}

As we have seen before, the invariant measure of the Arnold map with noise
is a fixed point of the transfer operator $L$ associated to the system.
Ulam's method can be employed to get a finite-dimensional approximation of
the system by a Markov chain and then use this to get information on the
system's invariant measure.

Let us recall briefly the Ulam method applied to maps of the interval. Let $%
S:\quad \lbrack 0,1]\longrightarrow \lbrack 0,1]$ be a map and $L$ the
corresponding transfer operator. Let us divide the interval $[0,1]$ in $N$
subintervals $I_{i}=\left[ x_{i-1}-x_{i}\right] $ and $\chi _{I_{i}}$ be the
associated characteristic function. Let $L_{N}$ be the finite dimensional
representation of the transfer operator corresponding to $S$ for the
assigned partition. The entries of $L_{N}$ are defined by $L_{N,i,j}=m\left(
S^{-1}(I_{i})\cap I_{j}\right) /m(I_{i})$, where $m$ is the Lebesque measure
($m(I_{i})=1/N$). The matrix $L_{N}$ is a stochastic matrix with nonnegative
entries and therefore possesses a non-negative left eigenvector with
eigenvalue equal to 1 representing a finite dimensional approximation of the
invariant measure \cite{ulam1960,LM,ding2001}.

Our computer aided estimates relies on a certified version of the Ulam
method described in \cite{GMN} where the probabilities $L_{N,i,j}$ are
estimated by computing with a rigorous bound the images of small intervals
by $S$. A faster method to get approximation of the system can be
implemented using a Monte Carlo approach. This method cannot lead to
certified estimates of course but it is interesting to compare the results
of this method and the rigorous one. To estimate $L_{N,i,j}$, following \cite%
{bose2001}, \cite{murray2004} we randomly select $w_{k,i},\ k=1,2,..M$ , $M $
points from the interval $I_{i}$ and $r_{i,j}$ be the number of points of
the image of the set $\left\{ w_{k,i}|w_{k,i}\in I_{i},\ k=1,2,..M\right\} $
under $S$ falling in the interval $I_{j}$. Then, $L_{N,i,j}$ is estimated as

\begin{equation}
L_{N,i,j}=\dfrac{r_{i,j}}{M}\approx \dfrac{m\left( S^{-1}(I_{i})\cap
I_{j}\right) }{m(I_{i})}.
\end{equation}

The same procedure can be done if the trajectories of the randomly selected
points are generated by a random dynamical system instead of $S$. This is
the method implemented in our Ulam Monte Carlo approximation.

One way to approximate non rigorously the invariant density is to iterate a
uniform starting density with the operator $L_{N}.$ More precisely we
implement the following steps \textit{i}) Select a nonnegative vector $v\in
R^{n}$, for instance $v=(1,1,...,1)^{T}$; \textit{ii}) define $%
v_{j+1}^{T}=v_{j}^{T}P(N)$ and a norm $\Vert v_{j+1}-v_{j}\Vert $ (for
instance the canonical Euclidean norm in $R^{n}$); \textit{iii}) iterate
step \textit{ii}) by using as new initial vector $v_{j+1}$; \textit{iv})
repeat steps \textit{ii}) and \textit{iii}) up to find $\Vert
v_{j+1}-v_{j}\Vert <\delta $ (where $\delta $ is the desired tolerance).
Then, after normalizing the final $v$ ($\Vert v\Vert =1$), the corresponding
finite dimensional approximation of the invariant measure associated to the
map $S$ is given by $\varphi ^{\ast }=\sum_{j=1}^{N}v_{i}\chi _{I_{i}}$.

\subsection{Invariant measure from the simulation of long orbits}

Let us come to the description of the pure Monte Carlo approach we use to
determine the invariant measure. This method is based on the simulations of
very long orbits, that are employed to estimate the corresponding invariant
measure. The main steps of the implemented algorithm are the followings. 
\textit{i}) An initial condition $X_{0}(j)\in \lbrack 0,1],\
j=1,2,3,..N_{IC} $ is chosen and then, for a given noise realization, an
orbit is generated by $N_{IT}$ iterations of the map. The initial conditions
are chosen uniformly distributed in $[0,1]$. \textit{ii}) The interval $%
[0,1] $ is divided in $N_{bins}$ and a very efficient algorithm is employed
to determine, for each iteration, to which bin the corresponding value along
the orbit belongs. In this way a cumulative histogram containing $%
N_{T}=N_{IC}N_{IT}$ data is generated. \textit{iii}) Then, the histogram is
normalized to have a unit total area and the corresponding step-wise
function will be an approximation of the invariant measure.

\subsection{Results}

In this section we compare the results obtained with these three methods. We
show the results for $\tau =0.709$, $\epsilon =1.4$ and different noise
amplitude. For the Ulam's method the parameter were chosen as follow. The
interval $I=[0,1]$ was portioned in 10000 subintervals $I_{i}$. Then, for
each $I_{i}$ a random set of $M=600000$ points, $\left\{ w_{k,i}|w_{k,i}\in
I_{i},\ k=1,2,..M\right\} $, was generated and used to estimate $L_{N,i,j}$.
Instead, for the pure Monte Carlo approach, the chosen values of the
parameters were $N_{IC}=5\times 10^{4}$, $N_{IT}=5\times 10^{6}$ and $%
N_{bins}=10000$. In figure \ref{fig:Ulam} are plotted the results for two
different values of the noise. The data in the top panels corresponds to a
noise amplitude $\xi =0.01$, while that in the bottom ones to $\xi =0.1$.
The results reported in the left top panel show that there is a very good
agreement between the measures computed with the three approach. However, a
finer inspection of these data, indicates that this is not the case as shown
in the zooming of the curves reported in the right top panel. In particular,
the invariant measure computed with the certified method is very smooth,
while that obtained with the two Monte Carlo approaches exhibit
fluctuations. In particular the fluctuations are less pronounced for the
measure computed with the simulations of long orbits.

Note that a smoothing of the results could be obtained by averaging over
different noise realizations (the results presented here were obtained using
a single noise realization). In the bottom panels are shown the results for
the noise amplitude $\xi=0.1$ and they are qualitative similar to the case
with smaller noise amplitude.

\begin{figure}[h]
\caption{Comparison of the invariant measures computed using Ulam's Monte
Carlo algorithm (blue line) with the simulations of long orbits (red line)
and the certified Ulam method (black line). Results plotted for $\protect%
\epsilon=1.4$ and $\protect\tau=0.709$. Top panels -- noise amplitude $%
\protect\xi=0.01$; and bottom panels -- $\protect\xi=0.1$. Note that in the
left panels the three curves appear to coincide. The right panels show a
zoom, in which smaller-scale differences between the three curves are
visible.}
\label{fig:Ulam}\centering
\includegraphics[scale=0.5]{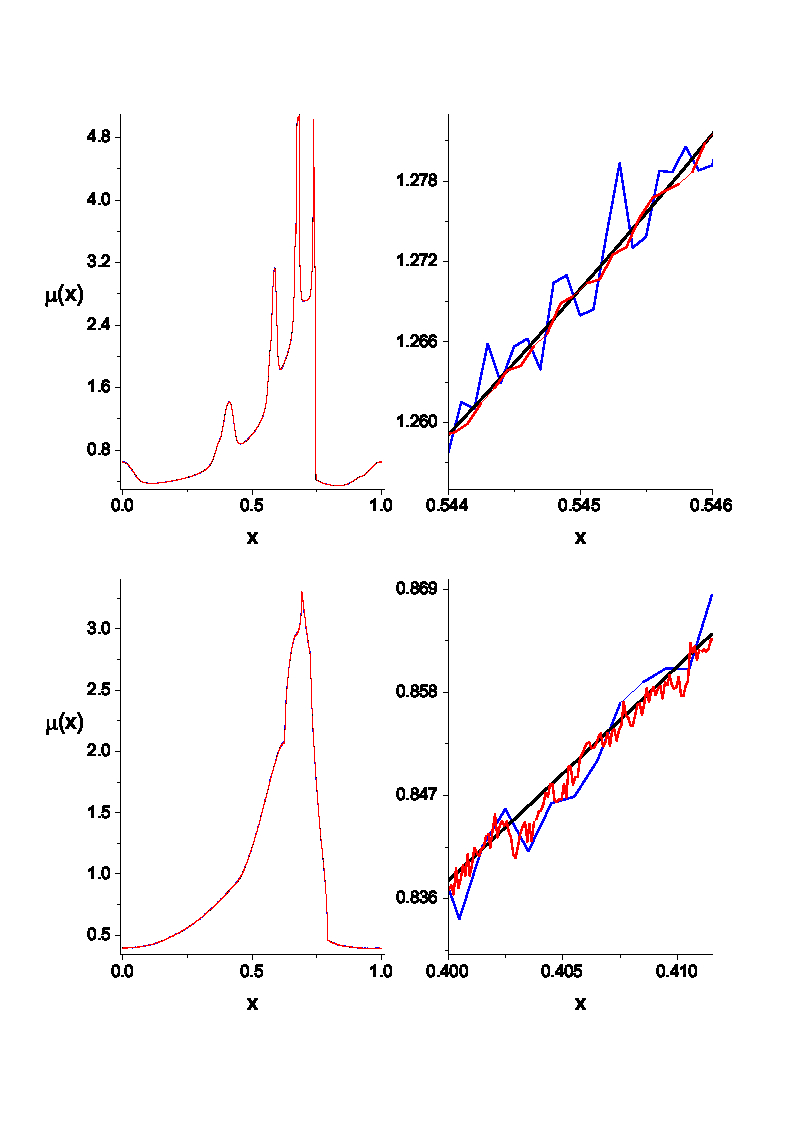}
\end{figure}

\subsection{Noise dependence of the rotation number's monotonicity}

\label{last}

In this section, we show numerical results that --- while not rigorous ---
suggest that, as remarked in Section~\ref{exp2}, the non-monotonicity of the
rotation number is a phenomenon typical of small noise amplitudes that
disappears as the noise amplitude is large enough. We consider the behavior
of the rotation number in the interval of $\tau $ values $[0.707,0.716]$. We
considered two amplitude of the noise: $\xi _{1}=0.01$ for which we known
that the rotation number is non monotonic and $\xi _{2}=0.05$. The numerical
simulations were carried out by using for each pair $(\tau ,\epsilon)$, with 
$\epsilon =1.4$ and $\tau \in \left\{ 0.707,0.708,....,0.716\right\} $,
60000 independent noise realizations. Moreover, for each realization, the
corresponding rotation number was estimated after $10^{6}$ iterates of the
Arnold map. The corresponding results are reported in figure~\ref{fig:noise}
and show that the rotation number becomes a monotonic function of $\tau $
when the noise amplitude is $\xi _{2}=0.05$.

We have seen that the presence of the noise promotes several effects on the
dynamics of the corresponding noisy Arnold map. In particular, the results
of the simulations discussed in Sections~\ref{sec:Ulam} and \ref{last} leads
to the following remarks:\newline
a) The presence of noise determines an increase of the support of the
corresponding invariant measure and to its smoothing.\newline
b) The comparison of the measures computed with the three approaches 
shows good agreement among them; see left panels of figure~\ref{fig:Ulam}.
However, the inspection of these data on a finer scale, shows that the
measures obtained with the Monte Carlo approaches exhibit fluctuations with
respect to that computed with the certified method.\newline
c) The noise amplitude strongly impacts the monotonicity properties of the
rotation number: i.e., for sufficiently high noise amplitude, a transition
from non-monotonic to monotonic behavior is observed.

To complete the above discussion, in figure \ref{fig:nonoise} are reported
the values of the rotation number against $\tau $ ($\epsilon =1.4$) in the
case of zero noise and for a larger interval of $\tau $ values. These data
were obtained using a Monte Carlo approach in which, for each $\tau $ value,
the mean and standard deviation of $\rho _{\tau }$ were obtained by
averaging over $10^{3}$ initial conditions (for each initial condition the
corresponding rotation number was estimated after $10^{6}$ iteration of the
noiseless map). The data clearly show the non monotonic behavior of the
rotation number in the chosen interval of $\tau $ values. The data on the
right of the figure also show that non monotonic behavior also occurs in
other intervals of $\tau $ values (of smaller amplitude). For each $\tau $
the distribution of the corresponding $\rho _{\tau }$ values determined with
this approach follow approximately a Gaussian distribution (data not shown).
Then, using the standard t-test, it was found that the mean probability that
the rotation number is monotonic in the intervals of interest is much lower
than $10^{-3}$.

\begin{figure}[h]
\centering
\includegraphics[scale=0.4]{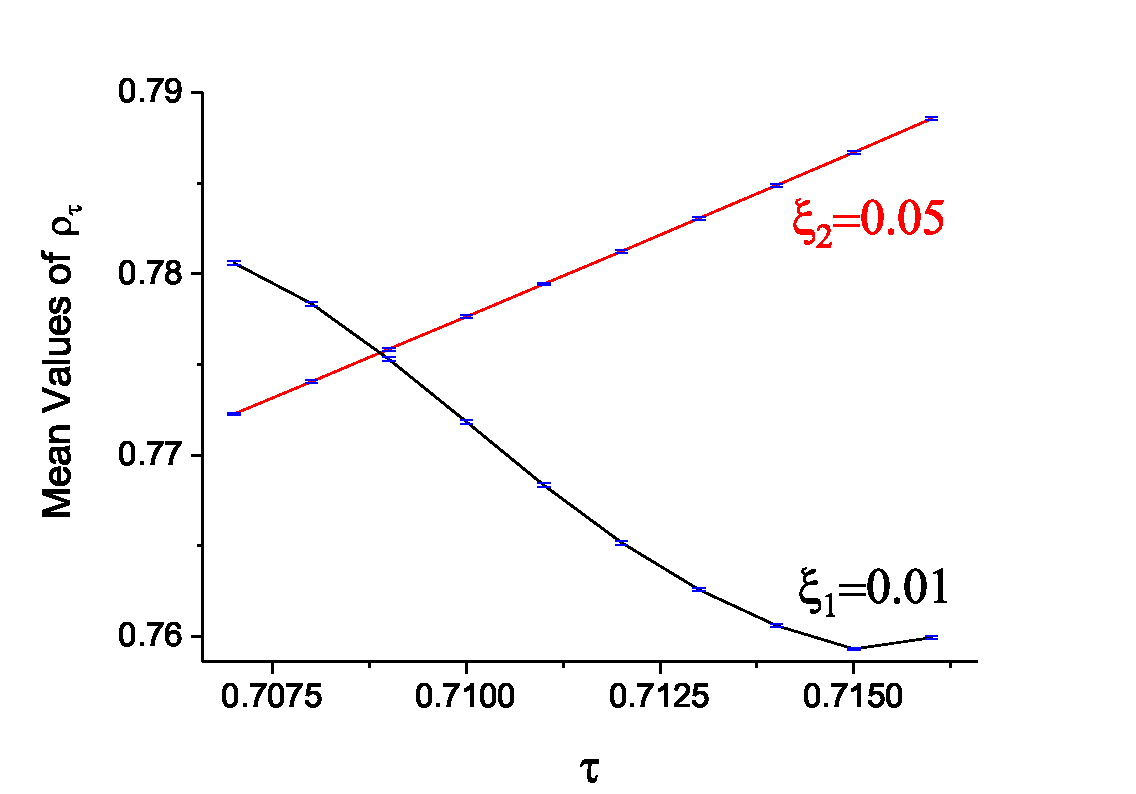}
\caption{Effect of the noise amplitude on the monotonicity of the rotation
number $\protect\rho = \protect\rho_{\protect\tau}$, shown for $\protect%
\epsilon=1.4$. The black line represents the numerical results for $\protect%
\xi=0.01$, while the red line is for $\protect\xi=0.05$. The errors bars
represent the corresponding standard deviations.}
\label{fig:noise}
\end{figure}

\begin{figure}[h]
\centering
\includegraphics[scale=0.4]{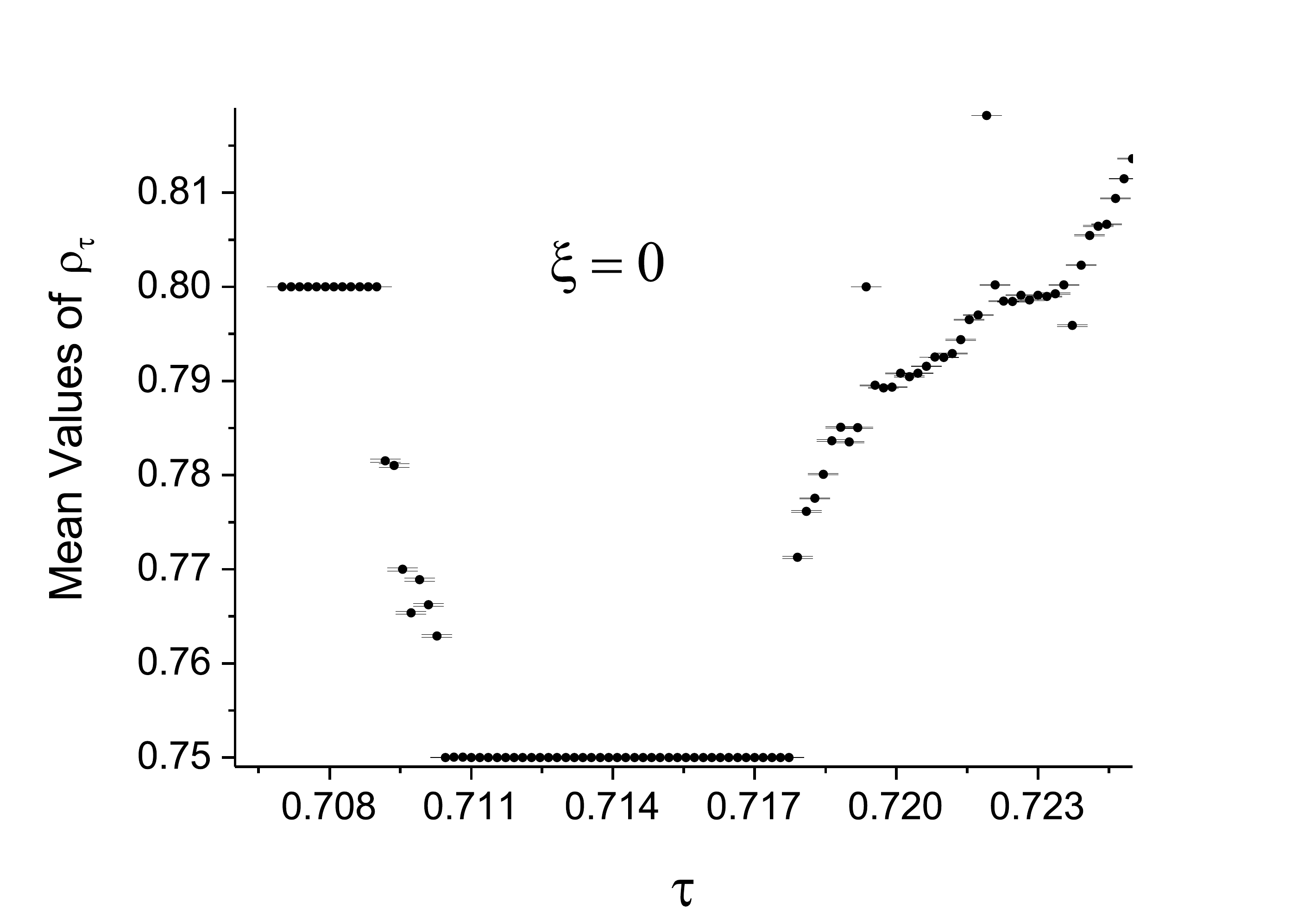}
\caption{Non monotonicity of the rotation number $\protect\rho_{\protect\tau%
} $ for $\protect\epsilon=1.4$ and in absence of noise ($\protect\xi=0$).
The black dots represent the mean values of $\protect\rho_{\protect\tau}$,
while the errorbars represent the corresponding standard deviations.}
\label{fig:nonoise}
\end{figure}

\newpage

\section{Concluding remarks}

\label{sec:conclude} {\ In this paper, we have studied Arnold circle maps
with strong nonlinearity and additive noise. We have proven rigorously that
the rotation number $\rho _{\tau }$ is differentiable with respect to the
parameter }$\tau ${, whenever the stochastically perturbed map is mixing,
providing a formula for the derivative of $\rho _{\tau }$ with respect to }$%
\tau ${. Moreover, we have shown, using a computer-aided proof, that $\rho
_{\tau }$ is not necessarily monotonic for $\epsilon \geq 1.$}

As outlined in the introduction, an important area of application of such
maps is the study of the El Ni\~{n}o--Southern Oscillation (ENSO)
phenomenon, which greatly affects seasonal-to-interannual climate
variability. The computational tools used in this paper are not directly
applicable to high-end, high-resolution global climate models (GCMs). But
there are three ways in which the results of this paper may shed light on
the behavior of such models and of the climate system itself.

First, the climate sciences have long relied on a systematic use of a
hierarchy of models, from the simplest ones, such as Arnold circle maps \cite%
{GCS08} --- through intermediate ones, such as those used to obtain Fig.~\ref%
{fig:bleacher} \cite{JNG94, JNG96} --- and on to the most detailed GCMs \cite%
{Ghil.2001, GR00, Held.2005, Sch.Dick.1974}. By applying systematically
advanced statistical methods \cite{Ghil.SSA.2002} to the simulations
produced by models of increasing detail and resolution, on the one hand, and
to observational data sets, on the other, it is possible to infer properties
of models at the top of the hierarchy and of the climate system itself,
given dynamical insights obtained with simple and intermediate models \cite%
{Ghil.2001, GR00}.

Second, linear response theory has played a crucial role in the arguments
developed in the present paper. This theory has already been applied fairly
widely in the climate sciences (\cite{Luu},\cite{GhLu},\cite{HM}). Linear
response theory was proved to apply generally in random systems in which the
presence of noise plays has a regulatizing effect (\cite{HM},\cite{GG}). We
can expect, therefore, that several of the general ideas presented herein
can be applied not only to very highly idealized models, like the Arnold
circle map, but to intermediate climate models as well, in the presence of
additive noise.

Third, methods of data-driven model building have been used to derive
relatively simple models directly from observational data sets and from
simulations of high-end models \cite{DAH.Arctic.2018, MSM.2015, KKG.2009,
MCP.1996}. With the greatly increased recent interest in such models in the
era of big data, this avenue will permit the formulation of more
sophisticated data-driven models that will be amenable to study by the full
set methods described and used herein.

%

\section{Acknowledgments}

The rigorous computations presented in subsections \ref{exp1} and \ref{exp2}
were performed on the supercomputer facilities of the M\'{e}socentre de
calcul de Franche-Comt\'{e}. \newline
JS was supported by the European Research Council (ERC) under the European
Union's Horizon 2020 research and innovation program (grant agreement No.
787304).

The present paper is TiPES contribution \#4; this project has received
funding from the European Union's Horizon 2020 research and innovation
program under grant agreement No. 820970.


\begin{thebibliography}{99}
\bibitem{ADF} Antown F, Dragi\v{c}evi\'{c}.D, Froyland, G.: \emph{Optimal
linear responses for Markov chains and stochastically perturbed dynamical
systems}. J. Stat Phys 170, 6, pp 1051--1087 (2018)

\bibitem{arnold1965} Arnold, V.I.: \textit{Small denominators. I: Mappings
of the circumference onto itself}. AMS Translations, Ser. 2 \textbf{46},
213--284 (1965)

\bibitem{arnold1991cardiac} Arnold V.I.: \textit{Cardiac arrhythmias and
circle mappings}. Chaos: An Interdisciplinary J. Nonlin. Sci. \textbf{1},
20--24 (1991),

\bibitem{Arn83} Arnold, V.I.: \textit{Geometrical Methods in the Theory of
Differential Equations}. Springer, 334 pp. (1983)

\bibitem{BahSau} Bahsoun, W., Saussol, B.: \textit{Linear response in the
intermittent family: differentiation in a weighted $C^{0}$-norm}. Discrete
Contin. Dyn. Syst. \textbf{36}(12), 6657-6668 (2016)

\bibitem{BRS} Bahsoun, W., Ruziboev, M., Saussol, B.: \emph{Linear response
for random dynamical systems}. arXiv:1710.03706

\bibitem{BGNN} Bahsoun, W., Galatolo, S., Nisoli, I., Niu, X.: \ \emph{A
Rigorous Computational Approach to Linear Response}. Nonlinearity \textbf{31}%
(no.3), 1073--1109 (2018)

\bibitem{bailey2018circle} Bailey, M.P., Drks, G.,Sheldonm A.C.: Circle maps
with gaps: Understanding the dynamics of the two-process model for
sleep--wake regulation. European Journal of Applied Mathematics,
https://doi.org/10.1017/S0956792518000190 (2018),1-24.

\bibitem{Bak86} Bak, P.: \textit{The devil's staircase}. Physics Today 
\textbf{39}(12), 38--45 (1986)

\bibitem{BB82} Bak, P., and Bruinsma, R.: \textit{One-dimensional Ising
model and the complete devil's staircase}. Phys. Rev. Lett. \textbf{49},
249-251 (1982)

\bibitem{Ba1} Baladi, V.: \textit{On the susceptibility function of
piecewise expanding interval maps}. Comm. Math. Phys. \textbf{275}(3),
839-859 (2007)

\bibitem{BT} Baladi, V., Todd, M.: \textit{Linear response for intermittent
maps}. Comm. Math. Phys. \textbf{347}(3), 857-874 (2016)

\bibitem{BaSma} Baladi, V., Smania, D.: \textit{Linear response formula for
piecewise expanding unimodal maps}. Nonlinearity \textbf{21}(4), 677-711
(2008)

\bibitem{BB} Baladi, V.: \textit{Linear response, or else} {Proceedings of
the {I}nternational {C}ongress of {M}athematicians---{S}eoul 2014. {V}ol. {%
III}}, 525--545 (2014)

\bibitem{BBS} Baladi, V., Benedicks, \ M., Schnellmann, N.: \ \emph{%
Whitney-Holder continuity of the SRB measure for transversal families of
smooth unimodal maps}. Invent. Math. \textbf{201}, 773-844 (2015)

\bibitem{BKL} Baladi, V., Kuna, T., Lucarini, V.:\ \ \emph{Linear and
fractional response for the SRB measure of smooth hyperbolic attractors and
discontinuous observables}. Nonlinearity \textbf{30}, 1204-1220 (2017)

\bibitem{Barn12} Barnston, A.G., Tippett, M.K., L'Heureux, M., \ Li, S.,
DeWitt, D.G.: \emph{Skill of real-time seasonal ENSO model predictions
during 2002--11: Is our capability increasing?}. Bull. Am. Meteorol. Soc., 
\textbf{93}(5), 631--651 (2012).

\bibitem{batista2003mode} Batista, A.M., Sandro, E., Pinto, de S., Viana,
R.L.,Lopes, S.R.: Mode locking in small-world networks of coupled circle
maps. Physica A: Statistical Mechanics and its Applications \textbf{322},
118-128 (2003)

\bibitem{bose2001} Bose, C. J., Murray R.: The exact rate of approximation
in Ulam's method. Discrete and Continuous Dynamical Systems \textbf{7},
219-235 (2001)

\bibitem{Chang94} Chang, P., Wang, B., Li, T., Ji, L.: \textit{Interactions
between the seasonal cycle and the Southern Oscillation: Frequency
entrainment and chaos in intermediate coupled ocean-atmosphere model}.
Geophys. Res. Lett. \textbf{21}, 2817--2820 (1994)

\bibitem{Chang96} Chang, P., Ji, L., Li, T., Fl\"ugel, M.: \textit{Chaotic
dynamics versus stochastic processes in El Ni\~no-Southern Oscillation in
coupled ocean-atmosphere models}. Physica D \textbf{98}, 301--320 (1996)

\bibitem{CSG11} Chekroun, , M.D., Simonnet, E., Ghil M.: \textit{Stochastic
climate dynamics: Random attractors and time-dependent invariant measures}.
Physica D \textbf{240}(21), 1685--1700, (2011)

\bibitem{D} Dolgopyat, D.: \textit{On differentiability of SRB states for
partially hyperbolic systems}. Invent. Math. \textbf{155}(2), 389-449 (2004)

\bibitem{ding2001} Ding, J., Wang, Z.: \textit{Parallel Computation of
invariant measures}. Annals of Operations Research \textbf{103}, 283-290
(2001)

\bibitem{EYuTz05} Eisenman, I., Yu, L., Tziperman, E.: \textit{Westerly wind
bursts: ENSO tail rather than the dog?}. J. Climate \textbf{18}, 5224--5238
(2005)

\bibitem{FKS82} Feigenbaum, M.J., Kadanoff, L.P., Shenker, S.J.: \textit{%
Quasiperiodicity in dissipative systems: A renormalization group analysis}.
Physica D \textbf{5}, 370--386 (1982)

\bibitem{GG} Galatolo, S., Giulietti, P.: \emph{Linear response for
dynamical systems with additive noise}. Nonlinearity
32, n. 6 pp. 2269-2301 (2019)

\bibitem{Gpre} Galatolo, S.: \emph{Quantitative statistical stability and
speed of convergence to equilibrium for partially hyperbolic skew products}%
.J. \'{E}c. Polytech. Math. \textbf{5}, 377--405 (2018)

\bibitem{GMN} Galatolo, S., Monge, M., Nisoli, I.: \emph{Existence of noise
induced order, a computer aided proof}. arXiv:1702.07024

\bibitem{GP} Galatolo, S., Pollicott, M.: \textit{Controlling the
statistical properties of expanding maps}. Nonlinearity \textbf{30},
2737-2751 (2017)

\bibitem{Ghil.2001} Ghil, M.: \emph{Hilbert problems for the geosciences in
the 21st century}. Nonlin. Processes Geophys. \textbf{8}, 211--222 (2001)

\bibitem{Ghil.SSA.2002} Ghil, M., Allen, M.R., Dettinger, M.D., Ide, K.,
Kondrashov, D., Mann, M.E., Robertson, A.W., Saunders, A., Tian, Y., Varadi,
F., Yiou, P.: \textit{Advanced spectral methods for climatic time series}.
Rev. Geophys. \textbf{40}(1), doi:\url{10.1029/2000RG000092} (2002)

\bibitem{GCS08} M. Ghil, M., Chekroun, M.D., Simonnet, E.: \textit{Climate
dynamics and fluid mechanics: Natural variability and related uncertainties}%
. {Physica D} \textbf{237}, 2111--2126, doi:\url{10.1016/j.physd.2008.03.036}
(2008)

\bibitem{GhLu} Ghil, M. and V. Lucarini \emph{The physics of climate
variability and climate change} Rev. Mod. Phys., submitted, arXiv:1910.00583

\bibitem{GJ98} Ghil, M., Jiang, N.: \textit{Recent forecast skill for the El
Ni\~{n}o/Southern Oscillation}. {Geophys. Res. Lett.} \textbf{25}, 171--174
(1998)

\bibitem{GR00} Ghil, M., Robertson, A.W.: \textit{Solving problems with
GCMs: General circulation models and their role in the climate modeling
hierarchy}.In D. Randall (Ed.) {General Circulation Model Development: Past,
Present and Future}, Academic Press, San Diego, 285--325 (2000)

\bibitem{GZT08} Ghil, M., Zaliapin I., Thompson, S.: \textit{A delay
differential model of ENSO variability: Parametric instability and the
distribution of extremes}. Nonlin. Processes Geophys. \textbf{15}, 417--433
(2008)

\bibitem{GhilDCDS} Ghil, M. \emph{The wind-driven ocean circulation:
Applying dynamical systems theory to a climate problem}, Discr. Cont. Dyn.
Syst. -- A, 37(1), 189--228 (2017)

\bibitem{glass1991cardiac} Glass L.: Cardiac arrhythmias and circle maps- A
classical problem. Chaos: An Interdisciplinary Journal of Nonlinear Science 
\textbf{1}, 13--19 (1991)

\bibitem{HM} Hairer, M., Majda, A.J.: \emph{A simple framework to justify
linear response theory}. Nonlinearity \textbf{23}, 909--922 (2010)

\bibitem{Held.2005} Held, I.M.: \emph{The gap between simulation and
understanding in climate modeling}. Bull. Am. Meteorol. Soc. \textbf{86},
1609--1614 (2005)

\bibitem{keener1984global} Keener, J.P., Glass, L.: \textit{Global
bifurcations of a periodically forced nonlinear oscillator}. J. Math.
Biology \textbf{21}, 175--190 (1984)

\bibitem{DAH.Arctic.2018} Kondrashov, D., Chekroun, M.D., Yuan, X., Ghil,
M.: \textit{Data-adaptive harmonic decomposition and stochastic modeling of
Arctic sea ice}. in \textit{Nonlinear Advances in Geosciences}, A. Tsonis,
Ed., Springer Science \& Business Media, 179--206, doi:%
\url{10.1007/978-3-319-58895-7} (2018)

\bibitem{Jal95} Jiang, S., F.-F. Jin, and M. Ghil, \emph{\ Multiple
equilibria, periodic, and aperiodic solutions \qquad \qquad in a
wind-driven, double-gyre, shallow-water model, }J. Phys. Oceanogr., 25,
764--786. \ (1995)

\bibitem{JNG94} Jin, F.-F., Neelin J.D., Ghil, M.: \textit{El Ni\~{n}o on
the Devil's Staircase: Annual subharmonic steps to chaos}. Science \textbf{%
264}, 70--72 (1994)

\bibitem{JNG96} Jin, F.-F., Neelin J.D., Ghil, M.: \textit{El Ni\~{n}%
o/Southern Oscillation and the annual cycle: Subharmonic frequency locking
and aperiodicity}. Physica D \textbf{98}, 442--465 (1996)

\bibitem{MSM.2015} Kondrashov, D., Chekroun M.D., Ghil, M.: \textit{%
Data-driven non-Markovian closure models}. Physica D \textbf{297}, 33--55,
doi:\url{10.1016/j.physd.2014.12.005} (2015)

\bibitem{K} Korepanov, A.: \textit{Linear response for intermittent maps
with summable and nonsummable decay of correlations}. Nonlinearity \textbf{29%
}(6), pp. 1735-1754 (2016)

\bibitem{KKG.2009} Kravtsov, S., Kondrashov, D., Ghil, M.: \textit{Empirical
model reduction and the modelling hierarchy in climate dynamics and the
geosciences}. in \textit{Stochastic Physics and Climate Modelling}, T.
Palmer and P. Williams (Eds.), Cambridge Univ. Press, pp. 35--72 (2009)

\bibitem{LM} Lasota, A., Mackey, M.C.: \emph{Probabilistic Properties of
Deterministic Systems}. Cambridge University Press (1986)

\bibitem{Li2} Liverani, C.: \textit{Invariant measures and their properties:
A functional analytic point of view}. Dynamical systems. Part II, pp.
185-237, Pubbl. Cent. Ric. Mat. Ennio Giorgi, Scuola Norm. Sup., Pisa (2003)

\bibitem{Luu} Lucarini, V.: \emph{Stochastic perturbations to dynamical
systems: A response theory approach}. J. Stat. Phys. \textbf{146}(4),
774--786 (2012)

\bibitem{Latif94} Latif, M., Barnett, T.P., Fl\"{u}gel, M., Graham, N.E.,
Xu, J-S., Zebiak, S.E.: \textit{A review of ENSO prediction studies}. Clim.
Dyn. \textbf{9}, 167--179 (1994)

\bibitem{Lor63a} Lorenz, E. N., \emph{Deterministic nonperiodic flow.} J.
Atmos. Sci., 20, 130--141. (1963)

\bibitem{Lor63b} Lorenz, E. N., \emph{\ The mechanics of vacillation.} J.
Atmos. Sci., 20, 448--464. (1963)

\bibitem{Mac} MacKay, R.S.: \emph{Management of complex dynamical systems}.
Nonlinearity \textbf{31}(n.2), 52-64 (2018)

\bibitem{murray2004} Murray, R.: Optimal partition choice for invariant
measure approximation for one-dimensional maps. Nonlinearity \textbf{17},
1623-1644 (2004)

\bibitem{Neel+98} \ Neelin, J.D., Battisti, D.S., Hirst, A.C., et al.: 
\textit{ENSO theory}. {J. Geophys. Res.--Oceans}, \textbf{103}(C7),
14261--14290 (1998)

\bibitem{PalWil09} Palmer, T., Williams, P., (Eds.): \textit{Stochastic
Physics and Climate Modelling}. Cambridge University Press, Cambridge, UK,
(2009)

\bibitem{MCP.1996} Penland, C.: \textit{A stochastic model of Indo-Pacific
sea-surface temperature anomalies}. Physica D \textbf{98}, 534--558 (1996)

\bibitem{Pierinial} Pierini, S., M. Ghil and M. D. Chekroun,\emph{\
Exploring the pullback attractors of a low-order quasigeostrophic ocean
model: The deterministic case,} J. Climate, 29, 4185--4202 (2016) .

\bibitem{Phil90} Philander, S.G.H.: \textit{El Ni\~{n}o, La Ni\~{n}a, and
the Southern Oscillation}. Academic Press, San Diego, (1990)

\bibitem{PV} Pollicott, M., Vytnova, P.: \emph{Linear response and periodic
points}. Nonlinearity \textbf{29}(no.10), 3047--3066 (2016)

\bibitem{R} Ruelle, D.: \textit{Differentiation of SRB states}. Commun.
Math. Phys. \textbf{187}, pp. 227-241 (1997)

\bibitem{SG01} Saunders A., Ghil, M.: \textit{A Boolean delay equation model
of ENSO variability}. Physica D \textbf{160}, 54--78 (2001)

\bibitem{Sch.Dick.1974} Schneider, S.H., Dickinson, R.E.: \textit{Climate
modeling}. Rev. Geophys. Space Phys. \textbf{12}, 447--493 (1974)

\bibitem{JS} Sedro, J.: \emph{A regularity result for fixed points, with
applications to linear response}. arXiv:1705.04078

\bibitem{TJ02} Timmermann, A., Jin, F-F.: \textit{A nonlinear mechanism for
decadal El Ni\~{n}o amplitude changes}. Geophys. Res. Lett. \textbf{29}(1),
https://doi.org/10.1029/2001GL013369 (2002)

\bibitem{Tzip94} Tziperman, E., Stone, L., Cane, M., Jarosh, H.: \textit{El
Ni\~no chaos: Overlapping of resonances between the seasonal cycle and the
Pacific ocean-atmosphere oscillator}. Science \textbf{264}, 72--74 (1994)

\bibitem{Tzip95} Tziperman, E., Cane, M.A., Zebiak, S.E.: \textit{%
Irregularity and locking to the seasonal cycle in an ENSO prediction model
as explained by the quasi-periodicity route to chaos}. J. Atmos. Sci. 
\textbf{50}, 293--306 (1995)

\bibitem{Tuc} Tucker, W.: \emph{Validated Numerics A Short Introduction to
Rigorous Computations}. Princeton Univ. Press (2011)

\bibitem{ulam1960} Ulam, S.M.:: \emph{A Collection of Mathematical Problems}%
. Interscience Publisher NY, 1960

\bibitem{Ver98} Verbickas, S.: \textit{Westerly wind bursts in the tropical
Pacific}. {Weather} \textbf{53}, 282--284 (1998)

\bibitem{Viana} Viana, M.: \emph{Lectures on Lyapunov Exponents}. Cambridge
Studies in Advanced Mathematics 145, Cambridge University Press (2014)

\bibitem{wiggins2003} Wiggins, S.: \emph{Introduction to Applied Dynamical
Systems and Chaos}. Springer NY, (2003) (see theorem 21.6.16)

\bibitem{Zal10} Zaliapin, I., Ghil, M.: \textit{A delay differential model
of ENSO variability, Part 2: Phase locking, multiple solutions and dynamics
of extrema}. {Nonlin. Processes Geophys.} \textbf{17}, 123--135 (2010)

\bibitem{zz} Zhang, Z.: \emph{On the smooth dependence of SRB measures for
partially hyperbolic systems}. arXiv:1701.05253

\bibitem{ZH07} Zmarrou, H., Homburg., A.J.:\emph{\ Bifurcations of
stationary measures of random diffeomorphisms}. Ergodic Theory Dynam.
Systems \textbf{27}(5):1651--1692 (2007)

\bibitem{herman1977} Herman, M.R.: \emph{Mesure de Lebesgue et nombre de
rotation, in: Geometry and Topology (Proc. III Latin Amer. School of Math.,
Inst. Mat. Pura Aplicada CNPq, Rio de Janeiro, 1976)}. Lecture Notes in
Math., vol.597, Springer, Berlin, pp. 271\^{a}\euro ``-293 (1977)

\bibitem{katok1995} Katok, A., Hasselblatt, B.: \emph{Introduction to the
modern theory of dynamical systems}. Cambridge University Press, (1995)

\bibitem{matsumoto2012} Matsumoto, S.:\emph{\ Derivatives of the rotation
number of one parameter families of circle diffeomorphisms}. Kodai Mat. J. 
\textbf{35}:115--125 (2012)

\bibitem{villanueva2008} Luque, A., Villanueva, J.: \emph{Computation of the
derivatives of the rotation number for parametric families of circle
diffeomorphisms}. Physica D \textbf{237}:2599--2615 (2008)
\end{thebibliography}
\end{document}